\documentclass[11pt]{article}
\usepackage[pdfstartview=FitH, bookmarksnumbered=true,bookmarksopen=true, colorlinks=true, pdfborder=001, citecolor=blue, linkcolor=blue,urlcolor=blue]{hyperref}
\usepackage[round,authoryear]{natbib}
\usepackage{xcolor}
\usepackage{amsmath}
\usepackage{amsthm}
\usepackage{amssymb}
\usepackage{abstract}
\usepackage{mathrsfs}
\usepackage{graphicx}
\usepackage[title]{appendix}
\usepackage{cases}
\usepackage{tabularray}
\usepackage{caption}
\usepackage{csquotes}
\usepackage[margin=1in]{geometry}

\allowdisplaybreaks[1]
\newtheorem{theorem}{Theorem}[section]
\newtheorem{assumption}[theorem]{Assumption}

\newtheorem{corollary}[theorem]{Corollary}
\newtheorem{definition}[theorem]{Definition}

\newtheorem{lemma}[theorem]{Lemma}

\newtheorem{remark}[theorem]{Remark}
\numberwithin{equation}{section}
\DeclareMathOperator*{\maxi}{maximize\,}
\DeclareMathOperator*{\mini}{minimize\,}
\renewcommand{\d}{\mathrm{d}}
\newcommand{\F}{\mathcal{F}}

\newcommand{\B}{\mathcal{B}}

\newcommand{\E}{\mathbf{E}}
\renewcommand{\P}{\mathbf{P}}

\newcommand{\1}{\mathbf{1}}

\newcommand{\var}{\mathrm{var}}

    \title{Portfolio selection with costly information acquisition}
    \author{Zongxia Liang\thanks{Department of Mathematical Sciences, Tsinghua University, Beijing 100084, China; Email: \texttt{liangzongxia@tsinghua.edu.cn}}\ \ \ \ \  Shu Wang\thanks{Department of Mathematical Sciences, Tsinghua University, Beijing 100084, China; Email: \texttt{shu-wang24@mails.tsinghua.edu.cn}}\ \ \ \ \  Jianming Xia\thanks{State Key Laboratory of Mathematical Sciences, Academy of Mathematics and Systems Science, Chinese Academy of Sciences, Beijing 100190, China; Email: \texttt{xia@amss.ac.cn}}
}	

\begin{document}
    \maketitle
    \begin{abstract}
   We investigate joint optimization on information acquisition and portfolio selection within a Bayesian adaptive framework.
    The investor dynamically controls the precision of a private signal and incurs costs while updating her belief about the unobservable asset drift. 
    Controllable information acquisition fails the classical separation principle of stochastic filtering. We adopt functional modeling of control to address the consequential endogeneity issues, then solve our optimization problem through dynamic programming.
    When the unknown drift follows a Gaussian prior, the HJB equation is often explicitly solvable via the method of characteristics, yielding sufficiently smooth classical solution to establish a verification theorem and confirm the optimality of feedback controls. 
    In such settings, we find that the investor's information acquisition strategy is deterministic and could be decoupled from her trading strategy, indicating a weaker separation property. In some degenerate cases where classical solutions may fail, semi-explicit optimal controls remain attainable by regularizing the information cost.
    \end{abstract}

\section{Introduction}

    Continuous-time portfolio selection problems under partial information have been extensively studied over the past few decades. A key concept that ensures the tractability of such problems is the separation principle. Although this principle has been interpreted in various ways since its introduction by \citet{wonham1968}, we adopt the statement of \citet{georgiou2013} that problems of optimal control and state estimation can be decoupled.

    Building on this idea, \citet{detemple1986} and \citet{gennotte1986} pioneered the study of expected utility maximization with partial observation. In the standard setting involving drift uncertainty, many studies have employed finite-dimensional filtering techniques. By enlarging the dimension of state space, the partially observed problem can often be reformulated as an equivalent completely observed one. Notable examples of this approach include the Bayesian adaptive model (e.g., \citet{karatzas2001}, \citet{Bis2019}), the Hidden Markov model (e.g., \citet{sass2004}, \citet{Rieder2005}), and the Ornstein–Uhlenbeck model (e.g., \citet{rishel1999}, \citet{brendle2006}). In these frameworks, the portfolio process does not influence the filtering procedure, and thus, the separation principle holds naturally.

    A common feature of the aforementioned works is the assumption that the investor only observes public information, typically the stock price process, which significantly departs from real-world conditions. Later studies have introduced private information, most notably insider information (e.g., \citet{monoyios2009}, \citet{danilova2010}) and expert opinions (e.g., \citet{frey2012}, \citet{davis2013}). In these models, private information is exogenously provided to the investor, implying that its acquisition is both costless and uncontrollable. This feature again ensures the validity of the separation principle.

    However, a growing body of empirical evidence (see, e.g., \citet{gargano2018}, \citet{guiso2020}, \citet{gupta2020}) suggests that an investor's information processing capacity can substantially affect trading performance. This has motivated a new line of research into controllable information acquisition, particularly through the lens of rational inattention theory, see \citet{mackowiak2023} for a comprehensive review. 
    Building on this idea, \citet{Andrei2020} introduced a model in which the investor receives a private signal about the unobservable predictive coefficient $\{\beta_t\}_{0 \le t \le T}$ of the risky asset, described by the following process:
    \begin{equation}\label{AHmodel}
        \d s_t=\beta_t\d t+\frac{1}{\sqrt{a_t}}\d B_{s,t}, \quad 0\le t\leq T.
    \end{equation} 
    Here, the information acquisition strategy $\{a_t\}_{0 \le t \le T}$ determines the precision of the signal, which the investor can control by choosing the intensity with which she acquires independent unit-precision signals (i.e., by adjusting her attention to news). It is worth noting, however, that model \eqref{AHmodel} excludes the critical case of complete information ignorance ($a_t = 0$), which is not only economically unsatisfactory but also creates unnecessary mathematical complications.

    

    In this paper, we investigate the joint optimization of information acquisition and portfolio selection under the Bayesian adaptive model, where the unobservable drift is modeled as a static random variable with a known prior distribution. The investor receives a private news signal described by \eqref{signal_process}, a normalized version of \eqref{AHmodel}, which accommodates the possibility of zero information acquisition. The cost of acquiring such information is linked to its precision via an information cost function, and this cost is directly deducted from the investor’s overall wealth. Assuming an investor who maximizes her expected utility of terminal wealth, our objective is to quantify her optimal level of costly information acquisition and to analyze the interplay between optimal trading and active learning.

    The fundamental idea of finite-dimensional filtering remains applicable in our setting. We show that, under some appropriately constructed probability space, the filtered dynamics of the investor’s wealth process, augmented by two auxiliary state variables, define a completely observed optimal stochastic control system. However, violation of the separation principle is inevitable. Consequently, in this controlled system, the driving Brownian motion, the underlying filtration, and even the probability space itself become control-dependent. 
    

    Besides, there is also a circular dependence between controls and observations, as the observation filtration, to which admissible controlled processes must adapt, is itself determined by the investor’s information acquisition strategy. To address these endogeneity issues, we propose the functional modeling: treating control as realization of progressively measurable functional by the observed processes. As such, we identify functional-form filtered dynamics on control-dependent probability spaces with those on a common auxiliary space, preserving the joint distribution of controls and driving Brownian motions (and thus the value function). This transformation reduces the original problem to a standard optimal stochastic control problem on the auxiliary space, which can be solved via dynamic programming techniques.
    
    
    The Hamilton–Jacobi–Bellman (HJB) equation associated with our problem is fully nonlinear and degenerate due to the introduction of controllable information acquisition, rendering general existence and uniqueness results largely intractable. A recent breakthrough by \citet{cohen2025} demonstrates that, for a broad class of control problems in which the observed process are of the form
    $$\d Y^u_s=\lambda^\top b(s,Y^u_s,u_s)\d s+\sigma(s,Y^u_s,u_s)\d W_s,\qquad Y^u_0=0,$$
    the value function can often be characterized as the unique viscosity solution to the corresponding HJB equation, with $\varepsilon$-optimal controls computable via numerical methods. However, their framework relies on relatively restrictive assumptions, including the boundedness of the support of the unobserved random variable $\lambda$, thereby excluding many important cases such as models with Gaussian priors. For this reason, we set aside the viscosity solution approach and instead pursue classical solutions in several special but inspiring settings.

    Focusing on CARA investors and normal priors (chosen for both mathematical tractability and statistical relevance), we show that for two broad classes of information cost functions, the original fully nonlinear second-order HJB equation can be reduced to either a quasi-linear free-boundary problem or a first-order Hamilton–Jacobi (HJ) equation. Both forms are amenable to analytical solution via the method of characteristics. The resulting value functions are sufficiently smooth to allow for a verification theorem, confirming the optimality of the feedback controls. In some degenerate cases where classical solutions may not exist, semi-explicit optimal controls remain attainable by regularizing the information cost. Furthermore, this whole approach extends naturally to CRRA investors and models with correlated noise.

    We note that existing works on controllable information acquisition, particularly the seminal paper of \citet{Andrei2020}, often overlook the aforementioned endogeneity issues. These studies typically focus on numerically solving the HJB equation while omitting discussions on the existence and uniqueness of its solution. Consequently, verification theorems for optimal feedback controls are rarely established. Our paper thus contributes to the literature by 
    \begin{itemize}
    \item[(1)] Developing a rigorous mathematical framework that reduces portfolio selection problems with endogenous information acquisition to standard optimal stochastic control;
    \item[(2)] Proposing an innovative solution technique based on the method of characteristics, which enables explicit solution of HJB equation and verification theorem of optimal feedback controls; and
    \item[(3)] Deriving novel economic interpretation that reveals the structural interaction between learning and trading:
    \end{itemize}

    We find that with normal priors, the investor’s optimal information acquisition strategy is deterministic, mirroring the behavior of the posterior variance process in the classical Kalman–Bucy filter. In fact, we prove that in order to achieve the optimal information-and-trading performance, the investor first selects her information acquisition strategy based solely on the posterior variance of the drift estimate, and subsequently adopts a portfolio proportional to the posterior mean (actually, proportional to the certainty equivalence value of the classical Merton fraction). This finding suggests another separation principle that the problems of optimal trading and active learning can be effectively decoupled, even though the optimization is conducted simultaneously.

    The remainder of this paper is organized as follows. Section \ref{sec2} introduces the functional modeling of information-and-trading strategies and rigorously formulate our optimization problem. In Section \ref{stoch_filt}, stochastic filtering is implemented, transforming the problem into a completely observed one within an auxiliary probability space. Sections \ref{truncated}–\ref{general} focus on explicit solutions to the HJB equation as well as verification results of optimal feedback control for two distinct types of information costs with markedly different properties. Section \ref{sec6} conducts numerical experiments and elaborates on the properties of the optimal information acquisition process. Finally, Section \ref{sec7} discusses a few extensions of the results from Section \ref{general}.

    \section{Problem formulation}\label{sec2}
    Let $(\Omega, \mathcal{F}, \{\mathcal{F}_t\}_{t \in [0,T]}, \P)$ be a filtered complete probability space, where $T > 0$ denotes the finite time horizon, and the filtration $\{\mathcal{F}_t\}_{t \in [0,T]}$, representing the full market information, satisfies the usual conditions. The financial market consists of two assets: a risk-free bond and a risky stock. The risk-free asset earns zero interest, while the stock price process $S = \{S_t\}_{t \in [0,T]}$ follows a geometric Brownian motion governed by the SDE:
        \begin{equation*}
        \frac{\d S_t}{S_t} = \mu \d t + \sigma \d W_t,\quad t\in[0,T],
        \end{equation*}
	where $W = \{W_t\}_{t \in [0,T]}$ is a standard Brownian motion with respect to (w.r.t.) the filtration $\{\mathcal{F}_t\}_{t \in [0,T]}$, the volatility $\sigma > 0$ is a known constant, and the drift $\mu$ is an unobservable random variable independent of $W$.
    We assume that $\mu$ follows a sub-Gaussian distribution, i.e., there exists $\eta>0$ such that 
        $$\E[e^{\eta\mu^2}]<+\infty.$$
    Define the revealing process $\widetilde W=\{\widetilde{W}_t\}_{t\in[0, T]}$ by $$\widetilde{W}_t=\mu t+\sigma W_t,\qquad t\in[0,T].$$
    Clearly, $S$ and $\widetilde{W}$ generate the same natural filtration. Therefore, $\widetilde{W}$ reveals to the investor all the information about the true value of $\mu$ embedded in the asset price dynamics.

    The investor, however, does not have access to the whole information $\{\F_t\}_{t\in [0,T]}$, but instead observes from
    \begin{itemize}
        \item [(i)] the asset price process (public information), and
        \item [(ii)] another signal process (private information), which evolves according to the following SDE:
    \end{itemize}
    \begin{equation}\label{signal_process}
        \begin{cases}
            \d \widetilde{B}_t =\mu\theta(t,\widetilde{W},\widetilde{B}) \d t +  \d B_t,\quad t\in[0,T],\\
            \widetilde B_0=0,
        \end{cases}	
		\end{equation}
    Here, $B = \{B_t\}_{t \in [0,T]}$ is another standard
        $\{\mathcal{F}_t\}_{t \in [0,T]}$-Brownian motion independent of both $W$ and $\mu$, and $\theta$ is the functional modeling of the investor's \emph{information acquisition strategy}. 

        To formally define admissible information acquisition strategies, we consider the space $C([0,T])$ consisting of all continuous real-valued functions on $[0,T]$, equipped with the supremum norm $$||\omega||_{C([0,T])}=\underset{t\in[0,T]}{\sup}|\omega(t)|.$$
        Let $\{\mathcal{B}_t\}_{t\in[0,T]}$ be the natural filtration generated by the coordinate mapping process:
        $$\B_t=\sigma\left(\{\omega\in C([0,T]): (\omega(t_1),\cdots,\omega(t_n))\in A, n\ge 1, t_1,\cdots,t_n\in [0,t], A\in \B(\mathbb{R}^n)\}\right).$$ 
        We also define the right-limit filtration by $\B_{t+}=\underset{s\in (t,T]}{\cap}\B_s,\;t\in [0,T)$, and $\B_{T+}=\B_T$. 
         

        \begin{definition}\label{theta}
        An admissible information acquisition strategy is a functional $$\theta:[0,T]\times C([0,T])\times C([0,T])\rightarrow [0,+\infty)$$ that satisfies the following conditions:
        \begin{itemize}
            \item[(a)] $\theta$ is progressively measurable w.r.t. the filtration $\{\B_{t+}\otimes\B_{t+}\}_{t\in[0,T]}$.
            \item[(b)] SDE \eqref{signal_process} admits a unique solution $\widetilde{B}=\{\widetilde{B}_t\}_{t\in [0,T]}$, which is progressively measurable w.r.t. $\F^{\mu,W,B}=
            \big\{\F^{\mu,W,B}_t\triangleq\sigma(\mu)\vee \F^{W,B}_t \big\}_{t\in[0,T]}$.
            \item[(c)] There exists a constant $C>0$ such that $\theta(t,\omega_1,\omega_2)\le C$ for all $\omega_1,\omega_2\in C([0,T])$.
        \end{itemize}
         \end{definition}
    Throughout this paper, we restrict attention to admissible information acquisition strategies.    
    Denote by $\{\F^{\widetilde{W},\widetilde{B}}_t\}_{t\in [0,T]}$ the natural filtration generated by $\widetilde{W}$ and $\widetilde{B}$, which we refer to as the observation filtration. This filtration captures all information available to the investor from both the public asset price and the private signal. Any stochastic process that is progressively measurable w.r.t. $\F^{\widetilde{W},\widetilde{B}}$ is called observable. 
    
    For notational convenience, we define the realization of $\theta$ by the observed processes $(\widetilde{W},\widetilde{B})$ as $$\vartheta\triangleq\{\vartheta_t=\theta(t,\widetilde{W},\widetilde{B})\}_{t\in [0,T]}.$$
    When $\vartheta_t>0$, the investor can divide the observed signal $\d\widetilde B_t$ by $\vartheta_t$ to obtain the following equivalent signal, similar to the formulation of \citet{Andrei2020}:
        $$\d \Xi_t=\mu\d t+\frac{1}{\vartheta_t}\d B_t.$$
    Heuristically, $\frac{\d {\Xi}_t}{\d t}$ provides an unbiased estimate of $\mu$, with instantaneous variance $\frac{1}{\vartheta_t^2\d t}$. Therefore, its precision is $\vartheta_t^2\d t$, the reciprocal of its variance. We assume that the cost of acquiring such information, deducted directly from the investor’s overall wealth, is determined by $k(\vartheta_t^2)\d t$, where $k:[0,+\infty)\rightarrow[0,+\infty]$ is the information cost function.
    
    However, when $\vartheta_t=0$, the expression $\frac{1}{\vartheta_t}\d B_t$ makes no sense. This motivates the adoption of the normalized convention \eqref{signal_process}, where the the magnitude of the unobserved noise term is fixed as a constant and the information acquisition strategy acts as an amplifier of the unknown drift. Note that this normalized formulation actually aligns with the classical framework of partially observed stochastic control problems, see, e.g.,  \citet[Section 2.7.6]{YongZhou1999}.
        
    We assume that $k$ is increasing, lower semicontinuous, and convex, with $k(0)=0$. The assumption $k(0)=0$ reflects the economic intuition that no cost is incurred when the investor chooses not to acquire private information, in which case she observes only the Gaussian white noise.

        Once an admissible information acquisition strategy is chosen, the investor selects an admissible trading strategy $\pi$, which is precisely defined as follows.
        
        \begin{definition}\label{pi}
        An admissible trading strategy is a functional $$\pi:[0,T]\times C([0,T])\times C([0,T])\rightarrow \mathbb{R}$$ that satisfies the following conditions:
        \begin{itemize}
            \item[(a)] $\pi$ is progressively measurable w.r.t. the filtration $\{\B_{t+}\otimes\B_{t+}\}_{t\in[0,T]}$.
            \item[(b)] There exist a constant $C>0$ and two bounded $\{\B_{t+}\otimes\B_{t+}\}_{t\in[0,T]}$-progressively measurable functionals $f$ and $g$ such that $$|\pi(t,\widetilde{W},\widetilde{B})|\le C+\sup_{s\in[0,t]}\left|\int_0^sf(\tau,\widetilde{W},\widetilde{B})\d\widetilde{W}_\tau\right|+\sup_{s\in[0,t]}\left|\int_0^s g(\tau,\widetilde{W},\widetilde{B}) \d\widetilde{B}_\tau\right|, \quad\forall\, t\in [0,T].$$ 
        \end{itemize} 
         \end{definition}
     Let $\varpi=\{\varpi_t\triangleq\pi(t,\widetilde{W},\widetilde{B})\}_{t\in [0,T]}$ denote the realization of the admissible trading strategy $\pi$ by $(\widetilde{W},\widetilde{B})$. Condition (b) then implies $\int_0^T\varpi_t^2\d t<+\infty$ a.s..

    The pair $(\theta,\pi)$ is referred to as an \emph{information-and-trading strategy}. It is called admissible if both $\theta$ and $\pi$ are admissible, and we denote this by $(\theta,\pi)\in\mathcal{A}$. By \citet[Lemma 1.2.10]{YongZhou1999}, for each $(\theta,\pi)\in\mathcal{A}$, its realization $(\vartheta,\varpi)$ by $(\widetilde{W},\widetilde{B})$ is observable. 
    
    Each admissible information-and-trading strategy generates a self-financing wealth process $X=\{X_t\}_{t\in[0,T]}$ satisfying 
    \begin{equation}\label{wealth_dynamics}
			\d X_t = \varpi_t \mu \d t + \varpi_t \sigma \d W_t - k\left(\vartheta_t^2\right) \d t,\qquad t\in [0,T]. 
		\end{equation}
    Let $U:\mathbb{R}\rightarrow\mathbb{R}$ be a von Neumann-Morgenstern utility function.The investor’s objective is to maximize the expected utility of terminal wealth over all admissible information-and-trading strategies:
		\begin{equation}\label{optimization_problem}
			\begin{split}
				&\maxi_{(\theta,\pi)\in\mathcal{A}}\quad \E\left[U(X_T)|X_0=x\right]\\
				&\text{ subject to } \quad \eqref{signal_process}  \text{ and } \eqref{wealth_dynamics}.
			\end{split}	
		\end{equation}
       
        \section{Girsanov transformation and filtering}\label{stoch_filt}

        In the wealth dynamics \eqref{wealth_dynamics}, the drift term $\mu$ is unobservable. To eliminate this hidden component, we shall employ Girsanov transformation. Given an admissible information acquisition strategy $\theta$, the exponential process $\Psi=\{\Psi_t\}_{t\in [0,T]}$, given by 
		\begin{equation*}
			\Psi_t=\exp\left\{-\frac{\mu}{\sigma}W_t-\frac{\mu^2}{2\sigma^2}t-\int_0^t\mu\vartheta_s\d B_s-\int_0^t\frac{1}{2}\mu^2\vartheta_s^2\d s\right\},\quad t\in[0,T],
		\end{equation*}
		is in fact a martingale w.r.t. $\F^{\mu,W,B}$. 
        This allows us to define a new probability measure $\widetilde{\P}$, whose dependence on $\theta$ and $\Psi$ is omitted for notational simplicity, via the Radon–Nikodym derivative: \begin{align}\label{measure_trans}
		    \left.\frac{\d \widetilde{\P}}{\d\P}\right|_{\F^{\mu,W,B}_T}=\Psi_T.
		\end{align} 
        Because  $W$ and $B$ are independent $(\P,\F^{\mu,W,B})$-Brownian motions, the Girsanov theorem implies that the processes
        $$\widetilde{W}_t=W_t+\frac{\mu}{\sigma}t,\qquad \widetilde{B}_t=B_t+\int_0^t \mu\vartheta_s \d s,\quad t\in [0,T]$$ are independent $(\widetilde{\P},\F^{\mu,W,B})$-Brownian motions. 
       Moreover, as $\mu$ is $\F^{\mu,W,B}_0$-measurable, it follows from the independent increment property of Brownian motion that $\mu$ is independent of both $\widetilde{W}$ and $\widetilde{B}$ under $\widetilde{\P}$.

       \begin{lemma}\label{distribution_mu} The random variable $\mu$ has the same distribution under $\P$ and $\widetilde{\P}$.
		\end{lemma}
\begin{proof} Applying Bayes' formula (See \citet[Lemma 3.5.3]{Karatzas1991}),  we have for all $A\in \mathcal {B}(\mathbb{R})$,  
         \begin{align*}\E^{\widetilde{\P}}[\1_A(\mu)|\F_t^{\widetilde{W},\widetilde{B}}]&=\frac{\E^\P[\1_A(\mu)\frac{\d \widetilde{\P}}{\d\P}|\F_t^{\widetilde{W},\widetilde{B}}]}{\E^\P[\frac{\d \widetilde{\P}}{\d\P}|\F_t^{\widetilde{W},\widetilde{B}}]}\\
        &=\frac{\E^\P[\1_A(\mu)\E^\P[\frac{\d \widetilde{\P}}{\d\P}|\F^{\mu,W,B}_t]|\F_t^{\widetilde{W},\widetilde{B}}]}{\E^\P[\E^\P[\frac{\d \widetilde{\P}}{\d\P}|\F^{\mu,W,B}_t]|\F_t^{\widetilde{W},\widetilde{B}}]}=\frac{\E^\P[\1_A(\mu)\Psi_t|\F_t^{\widetilde{W},\widetilde{B}}]}{\E^\P[\Psi_t|\F_t^{\widetilde{W},\widetilde{B}}]}.\end{align*}
	Setting $t=0$ and noting that $\F_0^{\widetilde{W},\widetilde{B}}$ is trivial, the result is obtained immediately. \end{proof} 
    
		It is easy to verify that 
		$$\left.\frac{\d \P}{\d \widetilde{\P}}\right|_{\F^{\mu,W,B}_t}=	\Psi_t^{-1}=\exp\left\{\frac{\mu}{\sigma}\widetilde{W}_t-\frac{\mu^2}{2\sigma^2}t+\int_0^t\mu\vartheta_s\d \widetilde{B}_s-\int_0^t\frac{1}{2}\mu^2\vartheta_s^2\d s\right\},\quad t\in[0,T].$$ 
We characterize its observable version in the next lemma.
        \begin{lemma}
            For every $t\in [0,T]$, we have
            $$\E^{\widetilde{\P}}\left[(\Psi_t)^{-1}\left|\F^{\widetilde{W},\widetilde{B}}_t\right.\right]
            =
            F(Y_t,Z_t),$$
            where 
            $$Y_t=\frac{\widetilde{W}_t}{\sigma}+\int_0^t\vartheta_s\d \widetilde{B}_s,\quad Z_t=\frac{t}{\sigma^2}+\int_0^t\vartheta_s^2\d s,$$ 
            and $F(y,z)=\mathbf{E}\left[\exp\{y\mu-\frac{z\mu^2}{2}\}\right]$ is well-defined and continuous on $\mathbb{R}\times(-\eta,+\infty)$. In particular, if $\mu\sim N(\mu_0,\sigma_0^2)$, we have the closed-form expression:
            \begin{equation}\label{F:normal}
            F(y,z)=\frac{1}{\sqrt{z\sigma_0^2+1}}\exp\left\{\frac{y^2\sigma_0^2+2y\mu_0-z\mu_0^2}{2(z\sigma_0^2+1)}\right\},\quad z\in(-\frac{1}{\sigma_0^2},+\infty).
            \end{equation}
        \end{lemma}
        
\begin{proof}  Obviously, both $Y_t$ and $Z_t$ are $\F^{\widetilde{W},\widetilde{B}}_t$-measurable, $t\in[0,T]$. Because  $\mu$ is independent of $\F^{\widetilde{W},\widetilde{B}}$ under $\widetilde{\P}$, we have
            \begin{align*}
                \E^{\widetilde{\P}}\left[(\Psi_t)^{-1}\left|\F^{\widetilde{W},\widetilde{B}}_t\right.\right]&=\E^{\widetilde{\P}}\left.\left[\exp\left\{\mu Y_t-\frac{1}{2}\mu^2Z_t\right\}\right|\F^{\widetilde{W},\widetilde{B}}_t\right]\\
                &=\E^{\widetilde{\P}}\left.\left[\exp\left\{\mu y-\frac{1}{2}\mu^2z\right\}\right]\right|_{y=Y_t,z=Z_t}=F(Y_t,Z_t),
            \end{align*}  
            as desired. The calculation of $F$ for normal prior distribution is straightforward.
\end{proof}

        Under the equivalent probability measure $\widetilde{\P}$, the wealth process $X$ satisfies the dynamics \begin{equation*}
              \d X_t = \pi(t,\widetilde{W},\widetilde{B})\sigma \d \widetilde{W}_t- k\left( \theta(t,\widetilde{W},\widetilde{B})^2\right) \d t,\quad t\in [0,T].
              \end{equation*}  It is clear that $X_T$ is $\F^{\widetilde{W},\widetilde{B}}_T$-measurable. Therefore, the original objective function $\E[U(X_T)]$ can be reformulated under $\widetilde{\P}$ as
        \begin{align*}
            \E[U(X_T)]&=\E^{\widetilde{\P}}[U(X_T)\Psi_T^{-1}]=\E^{\widetilde{\P}}[\E^{\widetilde{\P}}[U(X_T)\Psi_T^{-1}|\F^{\widetilde{W},\widetilde{B}}_T]]\\
            &=\E^{\widetilde{\P}}[U(X_T)\E^{\widetilde{\P}}[\Psi_T^{-1}|\F^{\widetilde{W},\widetilde{B}}_T]]=\E^{\widetilde{\P}}[F(Y_T,Z_T)U(X_T)].
        \end{align*}
        The above analysis leads to an optimal stochastic control problem under the equivalent probability measure $\widetilde{\P}$, given by
    \begin{align}\label{6}
   \widetilde V(t,x, y,z)=\sup_{({\theta},\pi)\in\mathcal{A}}& \E^{\widetilde{\P}}\left[F({Y}_T,{Z}_T)U({X}_T)| {X}_t=x, {Y}_t=y, {Z}_t=z\right]\\ \nonumber
    \text{subject to }&
    \begin{cases}
     \d {X}_t = {\pi}(t,\widetilde{W},\widetilde{B})\sigma \d \widetilde{W}_t- k\left( {\theta}(t,\widetilde{W},\widetilde{B})^2\right) \d t,\\	
    \d {Y}_t=\frac{1}{\sigma}\d \widetilde{W}_t+{\theta}(t,\widetilde{W},\widetilde{B})\d \widetilde{B}_t,\\
    \d {Z}_t=(\frac{1}{\sigma^2}+{\theta}(t,\widetilde{W},\widetilde{B})^2)\d t.
    \end{cases}
    \end{align} 


    The above filtering procedure directly incorporates the investor's information acquisition strategy, thereby violating the classical separation principle of partially observed stochastic control problems. This leads to the weak formulation of stochastic control problem \eqref{6}, where the probability measure $\widetilde{\P}$, the driving Brownian motion $\widetilde{B}$, and the observation filtration $\F^{\widetilde{W},\widetilde{B}}$ all depend on the controlled variable $\theta$. To proceed, we introduce an arbitrary auxiliary probability space $(\widehat{\Omega},\widehat{\F},\widehat{\P})$, on which a two-dimensional standard Brownian motion $(\widehat{W},\widehat{B})$ is defined. Consider the optimal stochastic control problem on $(\widehat{\Omega},\widehat{\F},\widehat{\P})$ in the strong form:
    \begin{align}\label{eq:v:phat}
    \widetilde{V}(t,x, y,z)={}\sup_{({\theta},{\pi})\in{\mathcal{A}}} &\E^{\widehat{\P}}\left[F(\widehat{Y}_T,\widehat{Z}_T)U(\widehat{X}_T)| \widehat{X}_t=x, \widehat{Y}_t=y, \widehat{Z}_t=z\right]\\
    \text{ subject to }&
    \begin{cases}
    \d \widehat{X}_t =  {\pi}(t,\widehat{W},\widehat{B})\sigma \d \widehat{W}_t- k\left( {\theta}(t,\widehat{W},\widehat{B})^2\right) \d t,\\
    \d \widehat{Y}_t=\frac{1}{\sigma}\d \widehat{W}_t+{\theta}(t,\widehat{W},\widehat{B})\d \widehat{B}_t,\\
    \d \widehat{Z}_t=(\frac{1}{\sigma^2}+{\theta}(t,\widehat{W},\widehat{B})^2)\d t.
    \end{cases}\nonumber
    \end{align} 
    Clearly, each $({\theta},{\pi})\in\mathcal{A}$ uniquely determines the law of the state processes $(\widehat{X},\widehat{Y},\widehat{Z})$. Therefore, Problems \eqref{6} and \eqref{eq:v:phat} share the same value function $\widetilde{V}$.
    

    Consider the following optimal stochastic control problem:
    \begin{align}
    \label{Q_value_function}
    {V}(t,x, y,z)={}\sup_{(\widehat{\vartheta},\widehat{\varpi})\in\widehat{\mathcal{A}}} &\E^{\widehat{\P}}\left[F(\widehat{Y}_T,\widehat{Z}_T)U(\widehat{X}_T)| \widehat{X}_t=x, \widehat{Y}_t=y, \widehat{Z}_t=z\right]\\
   \nonumber\text{subject to} &
    \begin{cases}
    \d \widehat{X}_t =  \widehat{\varpi}_t\sigma \d \widehat{W}_t- k\left( \widehat{\vartheta}_t^2\right) \d t,\\
    \d \widehat{Y}_t=\frac{1}{\sigma}\d \widehat{W}_t+\widehat{\vartheta}_t\d \widehat{B}_t,\\
    \d \widehat{Z}_t=(\frac{1}{\sigma^2}+\widehat{\vartheta}_t^2)\d t,
    \end{cases}
    \end{align}
     where $\widehat{\mathcal{A}}$ denotes the set of all $\F^{\widehat{W},\widehat{B}}$-progressively measurable processes $(\widehat{\vartheta},\widehat{\varpi})$ such that $0\le \widehat{\vartheta}_t\le C$ a.s. and $|\widehat{\varpi}_t|\leq C+\sup\limits_{s\in[0,t]}|\int_0^s f(\tau,\widehat{W},\widehat{B})\d\widehat{W}_\tau|+\sup\limits_{s\in[0,t]}|\int_0^sg(\tau,\widehat{W},\widehat{B})\d\widehat{B}_\tau|$ a.s., for some $C>0$, for some bounded $\{\B_{t+}\otimes\B_{t+}\}_{t\in[0,T]}$-progressively measurable functionals $f$ and $g$, and for all $t\in[0,T]$.
    Obviously, if $(\theta,\pi)\in\mathcal{A}$, then its realization by $(\widehat{W},\widehat{B})$ lies in $\widehat{\mathcal{A}}$. Thus, Problem \eqref{Q_value_function} is a relaxation of Problem \eqref{eq:v:phat} and we have $\widetilde V\leq V$. The following theorem is obvious.
        \begin{theorem} \label{recovery}Suppose that $(\widehat{\vartheta}^*,\widehat{\varpi}^*)\in\widehat{\mathcal{A}}$ solves 
          Problem \eqref{Q_value_function}.  If $(\widehat{\vartheta}^*,\widehat{\varpi}^*)$ is a realization of an admissible information-and-trading strategy $(\theta^*,\pi^*)$ by $(\widehat{W},\widehat{B})$, then $(\theta^*,\pi^*)$ solves Problem \eqref{optimization_problem}.
        \end{theorem}
    
    \begin{remark}\label{primal_controlled_system}
        In most of the existing literature, observable stochastic control systems under the primal measure $\P$ are often obtained with the help of the innovation processes $M$ and $N$: $$\d M_t=\d \widetilde{W}_t-\frac{\mu_t}{\sigma}\d t,\qquad \d N_t=\d \widetilde{B}_t-\mu_t\vartheta_t\d t,$$
    where $\mu_t=\E[\mu|\F^{\widetilde{W},\widetilde{B}}_t]$ is the posterior mean of the unknown drift $\mu$. In our setting, let $$G(y,z)\triangleq\frac{F_y(y,z)}{F(y,z)},$$
    then an application of the Bayes formula similar to Lemma \ref{distribution_mu} gives $\mu_t=G(Y_t,Z_t)$. The corresponding controlled system under $\P$ becomes \begin{align}\label{7}
        \sup_{({\theta},\pi)\in\mathcal{A}} &\E\left[U({X}_T)| {X}_t=x, {Y}_t=y, {Z}_t=z\right]\\
        \nonumber\text{subject to }&
        \begin{cases}
        \d {X}_t = \varpi_t[G(Y_t,Z_t)\d t+\sigma \d M_t]-k(\vartheta_t^2)\d t,\\	
        \d {Y}_t=\frac{1}{\sigma}\d M_t+{\vartheta}_t\d N_t+(\frac{1}{\sigma^2}+\vartheta_t^2)G(Y_t,Z_t)\d t,\\
        \d {Z}_t=(\frac{1}{\sigma^2}+\vartheta_t^2)\d t.
        \end{cases}
        \end{align} 
    In this case, the innovation processes $M$ and $N$ remain control-dependent. Moreover, the presence of the nontrivial function $G$ further complicates the dynamics. For this reason, we choose to work under the equivalent measure $\widetilde{\P}$.
    \end{remark}

    Our optimization problem \eqref{6} is highly irregular due to the introduction of controllable and costly information acquisition. Some technical challenges include:
    \begin{itemize}
        \item[(a)] The problem operates in an incomplete market setting where the risk of $\widetilde{B}$ cannot be hedged through risky asset transactions.
        \item[(b)] The system dynamics are nonlinear in the controlled variables $(\vartheta,\varpi)$.
        \item[(c)] The objective function is generally non-concave in the state variables  $y$ and $z$.
        \item[(d)] The corresponding HJB equation 
        \begin{equation}\label{HJB}\begin{cases}
        \underset{\pi,\theta}{\sup}\left\{V_t-k(\theta^2)V_x+(\frac{1}{\sigma^2}+\theta^2)(V_z+\frac{1}{2}V_{yy})+\frac{1}{2}\pi^2\sigma^2V_{xx}+\pi V_{xy}\right\}=0,\\
    V(T,x,y,z)=F(y,z)U(x)\end{cases}\end{equation} 
    is degenerate due to the lack of a second-order term in $z$.
    \end{itemize}   
    
    Under additional assumptions, such as bounded support of $\mu$ and polynomial growth conditions on the utility function $U$, the value function $V(t,x,y,z)$ can be characterized as the unique viscosity solution of the HJB equation \eqref{HJB}; see \citet{cohen2025} for details. 
    
    However, in the subsequent sections, we focus on the special case of a Gaussian prior (whose support is not bounded), where the classical solution to the HJB equation can be explicitly solved, the optimal functional $(\theta^*,\pi^*)$ can be explicitly computed, and a verification theorem can be established to confirm its optimality. The optimal realized information acquisition $\vartheta^*$, which is of particular interest, can be obtained either analytically or numerically.

		\section{Truncated linear information cost}\label{truncated}
		
		In this section, we make the following assumption:
        \begin{assumption}\label{ass:TL}
        \begin{itemize}
            \item[(a)] The utility function $U$ exhibits constant absolute risk aversion (CARA): $U(x)=-\frac{1}{\gamma}e^{-\gamma x},\; x\in\mathbb{R}$, where $\gamma>0$ is a constant;
            \item[(b)] The drift parameter $\mu$ follows a normal prior distribution: $\mu\sim N(\mu_0,\sigma_0^2)$, where $\mu_0\in\mathbb{R}$ and $\sigma_0>0$ are constants;
            \item[(c)] The information cost function $k$ is truncated linear, given by $$k(x)=cx\1_{[0,\beta^2]}(x)+\infty\1_{(\beta^2,+\infty)}(x),\quad x\in[0,+\infty),$$ 
            where $c>0$ and $\beta>0$ are constants. 
        \end{itemize}  
        \end{assumption}
        We proceed to solve the stochastic control problem \eqref{Q_value_function} explicitly under Assumption \ref{ass:TL}. The HJB equation in this case can be reformulated as a two-phase free-boundary problem:
        \begin{equation}
        \begin{cases}  
        V_t+\frac{1}{\sigma^2}(V_z+\frac{1}{2}V_{yy})-\frac{V_{xy}^2}{2\sigma^2V_{xx}}=0,  & \text{if}\; V_z+\frac{1}{2}V_{yy}-cV_x<0,\\
				 V_t+(\frac{1}{\sigma^2}+\beta^2)(V_z+\frac{1}{2}V_{yy})-\frac{V_{xy}^2}{2\sigma^2V_{xx}}-c \beta^2 V_x=0, & \text{if}\; V_z+\frac{1}{2}V_{yy}-cV_x\ge 0,\\                V(T,x,y,z)=F(y,z)U(x),\end{cases}\label{fb:V}
		\end{equation}
        where $F$ is given by \eqref{F:normal}. We expect to find a classical solution $V(t,x,y,z)\in C^{1,2,2,1}([0,T]\times \mathbb{R}\times\mathbb{R}\times [0,+\infty))$. 
        
        Due to the separable structure induced by the CARA utility, we conjecture a solution of the form: $$V(t,x,y,z)=-\frac{1}{\gamma}e^{-\gamma x}\exp\{h(t,y,z)\}.$$
        Substituting this ansatz into \eqref{fb:V} yields the following free-boundary problem for $h$:
		\begin{align}\label{18}
			\begin{cases}  h_t+\frac{1}{\sigma^2}(h_z+\frac{1}{2}h_{yy})=0, & \text{if}\; h_z+\frac{1}{2}h_{yy}+\frac{1}{2}h_{y}^2+c\gamma >0,\\
				 h_t+(\frac{1}{\sigma^2}+\beta^2)(h_z+\frac{1}{2}h_{yy})+\frac{\beta^2}{2}h_{y}^2+\beta^2c\gamma=0 , & \text{if}\; h_z+\frac{1}{2}h_{yy}+\frac{1}{2}h_{y}^2+c\gamma\le 0,\\
                 h(T,y,z)=-\frac{1}{2}\log(z\sigma_0^2+1)+\frac{y^2\sigma_0^2+2y\mu_0-z\mu_0^2}{2(z\sigma_0^2+1)}.
                 \end{cases}
		\end{align}
        It is straightforward to check that $h_z+\frac{1}{2}h_{yy}+\frac{1}{2}h_{y}^2=0$ at $t=T$, and hence the first regime (inequality strict) applies near the terminal time.

        For notational convenience, define $$H(t,z)\triangleq \frac{1}{\sigma_0^2z+\frac{\sigma_0^2}{\sigma^2}(T-t)+1},\qquad  (t,z)\in [0,T]\times (-\frac{1}{\sigma_0^2},+\infty).$$ We are now prepared to present the explicit solution to the free-boundary problem \eqref{18}.
    \begin{theorem}\label{value_function}
         The free-boundary problem \eqref{18} admits a classical solution $h\in C^{1,2,1}([0,T]\times \mathbb{R}\times [0,+\infty))$, which is given by
        \begin{equation*}
			h(t,y,z)=P(t,z)+H(t,z)\left[\mu_0y+\frac{1}{2}\sigma_0^2y^2\right].
		\end{equation*}
        Here, the function $P:[0,T]\times[0,+\infty)\rightarrow\mathbb{R}$ is given by $$P(t,z)=\begin{cases}
             \frac{1}{2}\log H(t,z)
             -\frac{H(t,z)}{2}\left[\mu_0^2z+\frac{\mu_0^2-\sigma_0^2}{\sigma^2}(T-t)\right], & z\in[\delta(t)\vee 0,+\infty),\\
             P(t_1,\delta(t_1))+\beta^2c\gamma(t_1-t)
             -\frac{\mu_0^2}{2\sigma_0^2}\left[H(t_1,\delta(t_1))-H(t,z)\right] &\\
		\hphantom{\frac{1}{2}\log}-\frac{\beta^2+\frac{1}{\sigma^2}}{2\beta^2}\left[\log H(t_1,\delta(t_1)) -\log H(t,z) \right],& z\in [0,\delta(t)\vee0),
        \end{cases}$$
        where the free boundary is given by
        $$z=\delta(t)=-\frac{T-t}{\sigma^2}-\frac{1}{\sigma_0^2}+\sqrt{\frac{T-t}{2c\sigma^2\gamma}},$$
        and when $\delta(t)>0$, $t_1$ is the unique non-negative solution of the following  equation:
        $$(\frac{1}{\sigma^2}+\beta^2)(t_1-t)+z=-\frac{T-t_1}{\sigma^2}-\frac{1}{\sigma_0^2}+\sqrt{\frac{T-t_1}{2c\sigma^2\gamma}}.$$
        \end{theorem}
       \begin{proof} The key idea is to conjecture a solution of quadratic form: $h(t,y,z)=P(t,z)+Q(t,z)y+\frac{1}{2}R(t,z)y^2$, and reduce the second-order PDE \eqref{18} to a system of first-order PDEs for the coefficients $P$, $Q$ and $R$. The resulting system is then solved using the method of characteristics. The rigorous argument, however, is lengthy and  technical, and it is deferred to Appendix \ref{TL_proof}. \end{proof} 

    \begin{corollary}\label{cor:V}
        The free-boundary problem \eqref{fb:V} has a classical solution $V\in C^{1,2,2,1}([0,T]\times \mathbb{R}\times\mathbb{R}\times [0,+\infty))$:
        \begin{equation*}
			V(t,x,y,z)=-\frac{1}{\gamma}e^{-\gamma x}\exp\left\{P(t,z)+H(t,z)\left[\mu_0y+\frac{1}{2}\sigma_0^2y^2\right]\right\},
		\end{equation*} where $P(t,z)$ is given in Theorem \ref{value_function}.
    \end{corollary}

         Now we focus on finding the optimal information-and-trading strategy. The optimizers for HJB equation \eqref{HJB} have the feedback form
            \begin{equation*}
                \begin{cases}\pi(t,x,y,z)=-\frac{V_{xy}}{\sigma^2V_{xx}}=\frac{1}{\sigma^2\gamma}h_y=\frac{1}{\sigma^2\gamma}\frac{\mu_0+\sigma_0^2 y}{\sigma_0^2z+\frac{\sigma_0^2}{\sigma^2}(T-t)+1},\\
                \theta(t,x,y,z)=\beta\1_{\left[0,-\frac{T-t}{\sigma^2}-\frac{1}{\sigma_0^2}+\sqrt{\frac{T-t}{2c\sigma^2\gamma}}\right)}(z).\end{cases}
            \end{equation*}
        Plugging them back into the controlled system \eqref{Q_value_function}, we obtain
        $$\widehat{Z}_t=
		\frac{1}{\sigma^2}t+\beta^2(t\land t^*), \qquad \widehat{Y}_t=
		\frac{1}{\sigma}\widehat{W}_t+\beta\widehat{B}_{t\land t^*},$$
        where $t^*=\inf\{t\ge 0:\widehat{Z}_t\ge \delta(t)\}$ is the first time when the state process $(t,\widehat{Z}_t)$ hits the free boundary $z=\delta(t)$. In fact, $(t,\widehat{Z}_t)$ can hit the free boundary at most once, this is because the slope of the free boundary satisfies
        \begin{equation*}
        \delta'(t)=\frac{1}{\sigma^2}-\frac{1}{4c\sigma^2\gamma}\left(\frac{T-t}{2c\sigma^2\gamma}\right)^{-1/2}<\frac{1}{\sigma^2}\le \frac{\d \widehat{Z}_t}{\d t}.
        \end{equation*}
	To express $t^*$ explicitly, consider the following equation:
		$$t=\left(\frac{1}{\sigma^2}+\beta^2\right)^{-1}\left(-\frac{T-t}{\sigma^2}-\frac{1}{\sigma_0^2}+\sqrt{\frac{T-t}{2c\sigma^2\gamma}}\right).$$
        This equation has at most one non-negative solution. If such a solution exists, then it must be equal to $t^*$. Otherwise, we have $t^*=0$. By defining $$\widehat{\vartheta}^*_t={\theta}(t,\widehat{X}_t,\widehat{Y}_t,\widehat{Z}_t)\qquad\widehat{\varpi}^*_t={\pi}(t,\widehat{X}_t,\widehat{Y}_t,\widehat{Z}_t),$$ we construct a candidate solution to the relaxed optimal stochastic control problem \eqref{Q_value_function}:
        $$\widehat{\vartheta}^*_t=\beta\1_{[0,t^*]}(t),\quad
        \widehat{\varpi}^*_t=\frac{1}{\sigma^2\gamma}\frac{\sigma_0^2(\frac{1}{\sigma}\widehat{W}_t+\beta\widehat{B}_{t\land t^*})+\mu_0}{\sigma_0^2 (\frac{1}{\sigma^2}t+\beta^2(t\land t^*))+\frac{\sigma_0^2}{\sigma^2}(T-t)+1}.$$
        The following theorem confirms its optimality.
        \begin{theorem}\label{verification_theorem} 
        Suppose that Assumption \ref{ass:TL} holds.
        For all $(t,\omega_1,\omega_2)\in[0,T]\times C([0,T])\times C([0,T])$,  let
            \begin{align*}
               & \theta^*(t,\omega_1,\omega_2)=\beta\1_{[0,t^*]}(t),\\
               & \pi^*(t,\omega_1,\omega_2)=
                \frac{1}{\sigma^2\gamma}\frac{\sigma_0^2(\frac{1}{\sigma}\omega_1(t)+\beta\omega_2(t\land t^*))+\mu_0}{\sigma_0^2 (\frac{1}{\sigma^2}t+\beta^2(t\land t^*))+\frac{\sigma_0^2}{\sigma^2}(T-t)+1}.  
            \end{align*}
            Then $(\theta^*,\pi^*)\in\mathcal{A}$ is an optimal solution to Problem \eqref{optimization_problem}. Moreover,
              $(\varpi^*,\vartheta^*)=\{(\varpi^*_t,\vartheta^*_t)\triangleq(\pi^*(t,\widetilde{W},\widetilde{B}),\theta^*(t,\widetilde{W},\widetilde{B}))\}_{t\in [0,T]}$ is the unique optimal realized strategy. In particular, $\vartheta^*$ is non-random.
        \end{theorem}
     \begin{proof} See Appendix \ref{verif}. \ \end{proof}

     \begin{remark}\label{primal_controlled_system_2}
          In the case where the prior on $\mu$ is Gaussian, the classical Kalman–Bucy filtering theory applies. Given any fixed information acquisition strategy $\vartheta$, standard results (see, e.g., \citet[Proposition 6.14]{bain2009} or \citet[Theorem 12.1]{liptser1977statistics}) imply that under probability measure $\P$:
    $$\d\mu_t=\frac{\sigma_0^2}{Z_t\sigma_0^2+1}\left[\frac{1}{\sigma}\d M_t+\vartheta_t\d N_t\right],\qquad \mu|_{\F^{\widetilde{W},\widetilde{B}}_t}\sim N(\mu_t,\Sigma_t^2),$$
    where $\Sigma_t^2\triangleq \var(\mu|\F^{\widetilde{W},\widetilde{B}}_t)=\frac{\sigma_0^2}{\sigma_0^2Z_t+1}$ is the posterior variance. Therefore, the state process $$Z_t=\frac{1}{\Sigma_t^2}-\frac{1}{\sigma_0^2}$$ 
    is a shifted reciprocal of the posterior variance of $\mu$, while
    $$Y_t=\frac{1}{\sigma_0^2}[(\sigma_0^2Z_t+1)\mu_t-\mu_0]$$
    should be considered as a linear transformation of the posterior mean. 
     \end{remark}

    Economically, the free boundary $z=\delta(t)$ characterizes the investor's threshold for acquiring private information. At time $t$, if the posterior variance of $\mu$ is sufficiently large such that $Z_t<\delta(t)$, the investor optimally acquires the maximum possible amount of private information to rapidly reduce the uncertainty about the asset’s expected return. Conversely, if the posterior variance is relatively small and $Z_t\ge\delta(t)$, the marginal benefit of further information acquisition no longer justifies its cost, and the investor optimally refrains from observing the private signal $\widetilde{B}$. 
    
    After determining her optimal information acquisition strategy $\vartheta^*$, the investor allocates wealth in the risky asset based on the posterior mean of the asset’s expected return: $$\varpi^*_t=\frac{\sigma_0^2Z_t+1}{\sigma_0^2Z_t+\frac{\sigma_0^2}{\sigma^2}(T-t)+1}\cdot\frac{\mu_t}{\sigma^2\gamma}.$$
    The term $\frac{\mu_t}{\sigma^2\gamma}$ corresponds to the certainty equivalent of the classical Merton fraction, representing the optimal position of an investor who disregards future learning opportunities. The slope coefficient $$\frac{\sigma_0^2 Z_t+1}{\sigma_0^2 Z_t+\frac{\sigma_0^2}{\sigma^2}(T-t)+1}$$
    \begin{itemize}
    \item is equal to $1$ at terminal time, showing that certainty equivalence principle (see, e.g., \citet{gennotte1986}) holds when no additional information is available for active learning.
    \item is less than $1$ before terminal time, showing that uncertainty of drift estimation leads to much conservative investment decision. Moreover, the investor considers future expenses on information acquisition, thereby reducing her investment and reserving funds for future needs.
    \item is increasing in $Z_t\ge 0$. Because a larger value of $Z_t$ corresponds to a lower posterior variance, boosting the investor’s confidence in the drift estimate and motivating a larger investment in the risky asset. 
    \end{itemize}
    
    Finally, if $\delta(0)\le0$, the investor does not acquire any private information throughout the entire investment horizon. In this degenerate case, our result aligns with the findings of \citet[Section 4]{Bis2019}, where controllable information acquisition is not considered.

    \section{Convex information cost}\label{general}

     In this section, we maintain the setting with a CARA utility function and a normal prior distribution but extend the analysis to incorporate general convex information cost functions. Our goal is to obtain classical solutions to the HJB equation along with the associated optimal feedback controls. To this end, we impose the following assumptions:
    \begin{assumption}\label{ass:k:convex} 
        \begin{itemize}
        \item[(a)] $U$ and $\mu$ satisfy Conditions (a) and (b) of Assumption \ref{ass:TL}. 
         \item[(b)] $k: [0,+\infty)\to [0,+\infty)$ is twice continuously differentiable with $k''(x)>0$ for every $x\in (0,+\infty)$. Moreover, $k^\prime(0)\triangleq \underset{x\rightarrow 0+}{\lim}k'(x)=0,\;k^\prime(+\infty)\triangleq \underset{x\rightarrow +\infty}{\lim}k'(x)=+\infty$.   
         \end{itemize}
         \end{assumption}
    Under Assumption \ref{ass:k:convex}, $k$ is strictly increasing and strictly convex, exhibiting increasing marginal cost of information acquisition. In this case, the HJB equation \eqref{HJB} reduces to the following form:
    \begin{equation}\label{conjugate}
    V_t+V_xk^*\left(\left(\frac{V_z+\frac{1}{2}V_{yy}}{V_x}\right)^+\right)+\frac{1}{\sigma^2}\left(V_z+\frac{1}{2}V_{yy}\right)-\frac{V_{xy}^2}{2\sigma^2 V_{xx}}=0,
    \end{equation} 
     where $k^*:[0,+\infty)\rightarrow [0,+\infty)$ is the Legendre-Fenchel transform of $k$, also known as the conjugate function, given by $$k^*(y)=\underset{x\ge0}{\sup}\{xy-k(x)\},\quad y\in[0,+\infty).$$
     Note that $k^*$ is an increasing, non-negative and convex function, and Assumption \ref{ass:k:convex} ensures that $k^*\in C^2(0,+\infty)\cap C[0,+\infty)$.
     By the variable separating structure of the model (CARA utility), we adopt the ansatz:
     \begin{equation*}
         V(t,x,y,z)=-\frac{1}{\gamma}e^{-\gamma x}\exp\{h(t,y,z)\}.
     \end{equation*} 
    We also conjecture that the solution satisfies $\frac{V_z+\frac{1}{2}V_{yy}}{V_x}\ge 0$ and will verify it later (In fact, this condition amounts to $\Gamma_u\le0$ in Lemma \ref{lma:Gamma:V}).
    Substituting this ansatz into the HJB equation yields the following PDE for $h$:
   \begin{equation}\label{eq:h:convex}
     \begin{cases}h_t+\frac{1}{\sigma^2}\left(h_z+\frac{1}{2}h_{yy}\right)+\widetilde{k}^*\left(h_z+\frac{1}{2}h_{yy}+\frac{1}{2}h_y^2\right)=0,\\
    h(T,y,z)=-\frac{1}{2}\log(z\sigma_0^2+1)+\frac{y^2\sigma_0^2+2y\mu_0-z\mu_0^2}{2(z\sigma_0^2+1)},
    \end{cases}
    \end{equation}
    where $\widetilde{k}^*(x)=-\gamma k^*(-\frac{x}{\gamma})$, $x\in (-\infty,0]$, is an increasing, non-positive, and concave function.
        
	 To proceed, consider the following standard form truncated piecewise-linear information cost function ($n\ge 1$ is a fixed integer): 
        $$k(x)=\begin{cases} k_1x, &x\in [0,\frac{1}{n}],\\
        k_i(x-\frac{i-1}{n})+\frac{1}{n}\sum_{j=1}^{i-1}k_j, & x\in[\frac{i-1}{n},\frac{i}{n}],\;i=2,3,\cdots,n^2,\\
		+\infty, &x\notin [0,n],\end{cases}$$ 
        where $\{k_i\}_{i=1}^{n^2}$ is a strictly increasing sequence of positive constants. 
        For this information cost, similar to \eqref{18}, the following equation for $h$ is obtained: 
        \begin{align*}
		\begin{cases}  h_t+\frac{1}{\sigma^2}(h_z+\frac{1}{2}h_{yy})=0, & h_z+\frac{1}{2}h_{yy}+\frac{1}{2}h_{y}^2 \in (-k_1\gamma,+\infty),\\
         h_t+(\frac{1}{\sigma^2}+\frac{i}{n})(h_z+\frac{1}{2}h_{yy})+\frac{i}{2n}h_{y}^2+\frac{k_1+ \cdots +k_i}{n}\gamma=0, & h_z+\frac{1}{2}h_{yy}+\frac{1}{2}h_{y}^2 \in(-k_{i+1}\gamma,-k_i\gamma], \\
	 h_t+(\frac{1}{\sigma^2}+n)(h_z+\frac{1}{2}h_{yy})+\frac{n}{2}h_{y}^2+\frac{k_1+ \cdots +k_{n^2}}{n}\gamma=0, & h_z+\frac{1}{2}h_{yy}+\frac{1}{2}h_{y}^2\in (-\infty,-k_{n^2}\gamma],\\
     h(T,y,z)=-\frac{1}{2}\log(z\sigma_0^2+1)+\frac{y^2\sigma_0^2+2y\mu_0-z\mu_0^2}{2(z\sigma_0^2+1)},\end{cases}
	\end{align*}
    where $i=1,2,\cdots,n^2-1$. This constitutes a multi-phase free-boundary problem.
    
    Observe that the essential term $a-b$ in the application of Lemma \ref{characteristic_method} remains invariant across different phases. Therefore, based on the proof of Theorem \ref{value_function}, we anticipate that the solution retains the form:
    \begin{equation}\label{23}
    h(t,y,z)=P(t,z)+\frac{\mu_0y+\frac{1}{2}\sigma_0^2y^2}{\sigma_0^2z+\frac{\sigma_0^2}{\sigma^2}(T-t)+1}.
    \end{equation}	
    Because smooth convex functions can always be approximated by piecewise-linear ones, we conjecture that the same solution structure remains valid for general convex $k$ as well. 
    Substituting \eqref{23} into \eqref{eq:h:convex}, we derive the following equation for $P$:
    \begin{align*}\begin{cases}
          P_t(t,z)+\frac{1}{\sigma^2}P_z(t,z)+\frac{\sigma_0^2}{2\sigma^2}H(t,z)+\widetilde{k}^*\left(P_z(t,z)+\frac{\sigma_0^2}{2}H(t,z)+\frac{\mu_0^2}{2}H(t,z)^2\right)=0,\\
          P(T,z)=-\frac{1}{2}\log(\sigma_0^2z+1)-\frac{z\mu_0^2}{2(\sigma_0^2z+1)}.
          \end{cases}
    \end{align*}
    To further simplify, define a transformed function $$\Gamma(t,u)=P(t,u-\frac{1}{\sigma^2}(T-t))+\frac{1}{2}\log (\sigma_0^2u+1)-\frac{\mu_0^2}{2\sigma_0^2}\frac{1}{\sigma_0^2u+1}.$$
   The resulting PDE for $\Gamma$ becomes a first-order nonlinear Hamilton-Jacobi equation:\begin{equation}\label{29}\begin{cases}
        \Gamma_t(t,u)+\frac{\sigma_0^2}{2\sigma^2}\frac{1}{\sigma_0^2u+1}+\widetilde{k}^*(\Gamma_u(t,u))=0,\\
        \Gamma(T,u)=-\frac{\mu_0^2}{2\sigma_0^2}.
    \end{cases}\end{equation}
   The results above are summarized in the following lemma: 
    
    \begin{lemma}\label{lma:Gamma:V}
        Assume that $\Gamma(t,u)\in C^2([0,T)\times[0,+\infty))\cap C([0,T]\times[0,+\infty))$ is a classical solution to  \eqref{29} such that $\Gamma_u(t,u)\le 0$ holds for $(t,u)\in[0,T)\times[0,+\infty)$. Then the function
    \begin{equation*}
    \begin{split}
    V(t,x,y,z) =  -\frac{1}{\gamma}e^{-\gamma x}\sqrt{H(t,z)}\exp&\left\{
    \Gamma\left(t,z+\frac{1}{\sigma^2}(T-t)\right) + \frac{1}{2}H(t,z)\left(\sigma_0y+\frac{\mu_0}{\sigma_0}\right)^2
    \right\}
    \end{split}
    \end{equation*}
    belongs to $C^2([0,T)\times\mathbb{R}^2\times[0,+\infty))\cap C([0,T]\times\mathbb{R}^2\times[0,+\infty))$ and is a classical solution to the HJB equation \eqref{HJB}.
    \end{lemma}

    \begin{remark}\label{det_opti}
        The Hamilton-Jacobi equation \eqref{29} can be equivalently written as 
        \begin{equation}\label{trans_HJB}
        \begin{cases}
            {\Gamma}_t(t,u)+\frac{\sigma_0^2}{2\sigma^2}\frac{1}{\sigma_0^2u+1}+\underset{y>0}{\inf}\;\left\{y\cdot{\Gamma}_u(t,u)+\gamma k(y)\right\}=0,\\
            {\Gamma}(T,u)=-\frac{\mu_0^2}{2\sigma_0^2},
        \end{cases}
    \end{equation}
    which corresponds to the dynamic programming equation of the following deterministic control problem:
        \begin{align}\label{inner}
        &\Gamma(t,u)=\inf_{\vartheta}\;\left\{ \int_t^T\gamma k(\vartheta_s^2)\d s+\int_t^T\frac{\sigma_0^2}{2\sigma^2}\frac{1}{\sigma_0^2u_s+1}\d s-\frac{\mu_0^2}{2\sigma_0^2}\right\}\\
        &\nonumber\text{subject to }
         \quad\d u_s=\vartheta_s^2\d s,\;u_t=u.
        \end{align}
    We will revisit this problem and explore its economic interpretation in Section \ref{eco_int}.
    \end{remark}

    For now, we further assume that $k''(0)\triangleq\underset{x\rightarrow0+}{\lim}k^{\prime\prime}(x)>0$, which is equivalent to the regularity condition $\underset{x\rightarrow0+}{\lim}(k^*)^{\prime\prime}(x)<+\infty$.
    The irregular case $\underset{x\rightarrow0+}{\lim}k^{\prime\prime}(x)=0$ will be discussed later. 

    \begin{theorem}\label{general_characteristic}
        Under Assumption \ref{ass:k:convex} and the additional assumption $k''(0)>0$, the Hamilton-Jacobi equation \eqref{29} admits a unique classical solution $\Gamma(t,u)\in C^2([0,T)\times[0,+\infty))\cap C([0,T]\times[0,+\infty))$, which satisfies $\Gamma_u(t,u)\le 0$ for all $(t,u)\in[0,T)\times[0,+\infty)$.
    \end{theorem}
    \begin{proof}
        The proof essentially follows the standard method of characteristics and is deferred to Appendix \ref{GC_proof}. The regularity condition $k''(0)>0$ is crucial, as it ensures the twice continuous differentiability of $k^*$ at $x=0$, which in turn guarantees the continuous differentiability of the solution to the associated Hamilton's ODE \eqref{31} with respect to the initial data. \end{proof}

    \begin{theorem}\label{verification_2}
        Under Assumption \ref{ass:k:convex} and the additional assumption $k''(0)>0$, let $\Gamma$ be given by Theorem \ref{general_characteristic}. Then the following ODE
        \begin{align}\label{41}
            \d Z_t=\left(\frac{1}{\sigma^2}+(k')^{-1}\left(-\frac{1}{\gamma}\Gamma_u(t,Z_t+\frac{1}{\sigma^2}(T-t))\right)\right)\d t, \qquad Z(0)=0,
        \end{align} 
        admits a unique solution on $[0,T]$. Let 
        \begin{eqnarray}
            &&(\vartheta^*_t)^2=(k')^{-1}\left(-\frac{1}{\gamma}\Gamma_u(t,Z_t+\frac{1}{\sigma^2}(T-t))\right),  \label{42} \\
            &&\nonumber Y_t=\frac{1}{\sigma}\widetilde{W}_t+\int_0^t\vartheta^*_s\d \widetilde{B}_s,
        \end{eqnarray}
    where $\widetilde{B}$ is the signal process generated by $\vartheta^*$, and 
        \begin{equation}\label{43}
            \varpi^*_t=\frac{1}{\sigma^2\gamma}\frac{\mu_0+\sigma_0^2Y_t}{\sigma_0^2Z_t+\frac{\sigma_0^2}{\sigma^2}(T-t)+1}.
        \end{equation}
         Then $(\vartheta^*,\varpi^*)$ is the unique optimal realized information-and-trading strategy. The corresponding $(\widehat{\theta},\widehat{\pi})$ that generates $(\vartheta^*,\varpi^*)$ is optimal to Problem \eqref{optimization_problem}. In particular, $\vartheta^*$ is non-random.
        \end{theorem}
        
\begin{proof} See Appendix \ref{verif_proof}. \end{proof}
    
    We now turn to the irregular case $k''(0)=0$. In this case we cannot expect to obtain a $C^2$-classical solution to Hamilton-Jacobi equation \eqref{29} in general. Nevertheless, an optimal information-and-trading strategy can still be obtained by introducing a regularization term to the original information cost function: $$k^{(b)}(x)\triangleq k(x)+bx^2\quad (b>0).$$
    Now $k^{(b)}$ satisfies $\underset{x\rightarrow0+}{\lim}(k^{(b)})^{\prime\prime}(x)=b>0$ and Theorem \ref{verification_2} applies to $k^{(b)}$. Let $(\vartheta^{(b)},\varpi^{(b)})$ be the optimal information-and-trading strategy for cost function $k^{(b)}$.
    
    \begin{theorem}\label{limit}
        The limiting function $\vartheta^{(0)}_t\triangleq \underset{b\rightarrow 0}{\lim}\vartheta^{(b)}_t$ is well-defined and continuous
        for $t\in[0,T]$. Set
        $$Y^{(0)}_t=\frac{1}{\sigma}\widetilde{W}_t+\int_0^t \vartheta^{(0)}_s \d\widetilde{B}^{(0)}_s,\qquad Z^{(0)}_t=\frac{t}{\sigma^2}+\int_0^t (\vartheta^{(0)}_s)^2 \d s,$$
        where $\widetilde{B}^{(0)}$ is the signal process generated by information acquisition strategy $\vartheta^{(0)}$. Define $\varpi^{(0)}$ by \eqref{43} with $Y,Z$ replaced by $Y^{(0)},Z^{(0)}$, then $(\varpi^{(0)},\vartheta^{(0)})$ is the optimal realized information-and-trading strategy for information cost function $k$. 
    \end{theorem}

\begin{proof}  See Appendix \ref{limit_proof}.  \end{proof}
   \vskip 4pt 
    Note that this proof relies on a key fact established in Section \ref{eco_int}: If the investor fixes an arbitrary but \emph{non-random} information acquisition strategy $\vartheta^*$ and optimizes only over trading strategies $\varpi$ to maximize her expected utility of terminal wealth, then her optimal response is actually given by \eqref{43}. 

    \section{Numerical results: properties of optimal information acquisition}\label{sec6}

    In this section, we conduct numerical experiments to obtain approximation of the optimal information acquisition process. Assume that the cost function is given by $k(x)=cx^2$. Our analysis focuses on the qualitative properties of $\vartheta^*$and examines its sensitivity to key model parameters. Unless otherwise stated, the parameter values are set to the maximum likelihood estimates reported by \citet{Guan2025}, namely $c=0.002, \mu_0=0.172, \sigma_0=0.121, \sigma=0.192, \gamma=2, T=1$.
    
     We have mentioned that the Hamilton-Jacobi equation \eqref{29} can be expressed equivalently as a Hamilton-Jacobi-Bellman equation \eqref{trans_HJB}, for which various numerical finite difference schemes are well established. Let $h>0$ denote the spatial step size and $\tau>0$ the temporal step size. We adopt the standard upwind scheme: $$\frac{1}{\tau}\left[\Gamma(t_{n+1},u_j)-\Gamma(t_n,u_j)\right]+\frac{\sigma_0^2}{2\sigma^2}\frac{1}{\sigma_0^2u_j+1}+\widetilde{k}^*\left(\frac{1}{h}[\Gamma(t_n,u_{j+1})-\Gamma(t_n,u_j)]\right)=0,$$
     which has been shown to be both convergent and stable under mild conditions; see, e.g., \citet[Chapter 8]{Fleming2006} or \citet{Sun2015} and references therein. While most of those conditions are inherently satisfied by our model, the ratio $\tau/h$ must be carefully chosen to satisfy the CFL condition (Courant-Friedrichs-Lewy). Specifically, our analysis in the proof of Theorem \ref{general_characteristic} shows that the slope of the characteristic curves satisfies $$0\le \left|\frac{\Delta u}{\Delta t}\right|\le (k')^{-1}(\frac{\sigma_0^4 T}{2\sigma^2\gamma})$$ for $(t,u)\in [0,T]\times [0,+\infty)$. Thus, a sufficient condition for the CFL condition is $$\frac{h}{\tau}\ge (k')^{-1}(\frac{\sigma_0^4 T}{2\sigma^2\gamma}).$$
    
     Before presenting the numerical results, we first state several analytical properties of $\vartheta^*$, derived directly from the structure of the optimization problem:
    \begin{theorem}\label{property}
     Under Assumption \ref{ass:k:convex}, the optimal realized information acquisition strategy $\vartheta^*$ has the following properties:
\begin{itemize}
        \item[(a)] $\vartheta^*_t\neq 0$ for all $t\in [0,T)$ and $\vartheta_T^*=\underset{t\rightarrow T}{\lim}\vartheta_t^*=0.$

        \item[(b)] $\vartheta^*_t$ is a decreasing function of $t$.

        \item[(c)] $\vartheta^*$ does not depend on the specification of prior expectation $\mu_0$.
        \end{itemize}
    \end{theorem}
    
\begin{proof} See Appendix \ref{property_proof}. \end{proof} 
    
    These results suggest that the investor chooses to acquire more information early in the investment period, as earlier information can be exploited over a longer horizon. The independence of $\vartheta^*$ from $\mu_0$ indicates that the investor focuses on the absolute error (as opposed to relative error) of her drift estimate, with the posterior variance of the unknown drift being the exclusive determinant of her information acquisition strategy. 
    
    The effects of various model parameters on $\vartheta^*$ are exhibited in Figures \ref{a}-\ref{c}:
    \begin{itemize}
        \item Figure \ref{a} shows that the optimal level of information acquisition decreases as the cost parameter $c$ increases, which aligns with the economic intuition that higher price reduces demand.
        \item Figure \ref{b} demonstrates that higher risk aversion $\gamma$ leads to lower information acquisition. This is because more risk-averse investors allocate less wealth to risky assets, reducing the marginal benefit of acquiring costly private information.
        \item Figure \ref{c} illustrates that information acquisition increases with the prior variance $\sigma_0^2$ of the drift estimate, because greater prior uncertainty motivates investors to reduce ambiguity by acquiring more information.
    \end{itemize}
   
    \begin{figure}
		\centering
		\begin{minipage}{0.3\textwidth} 
			\centering
			\includegraphics[scale=0.22]{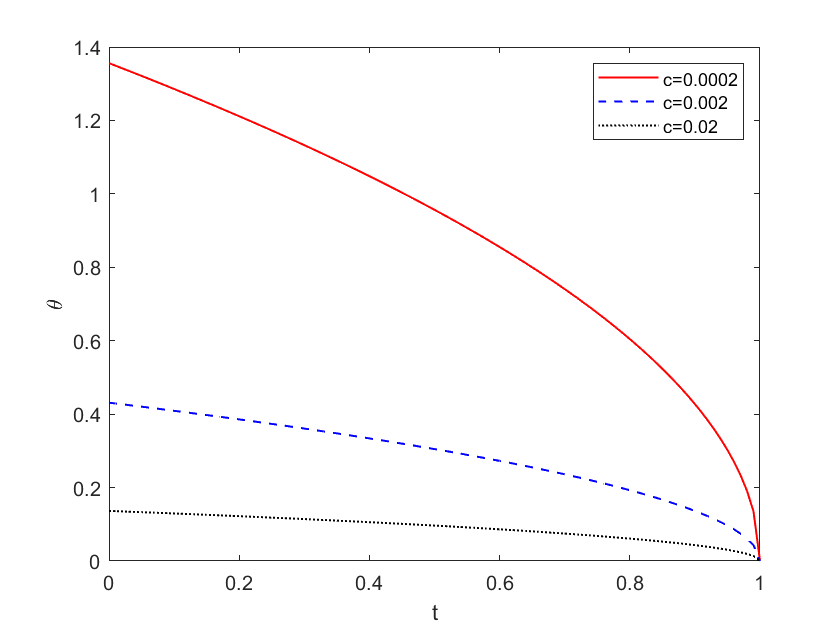}
            \caption{Effect of information cost}\label{a}
		\end{minipage}
		\begin{minipage}{0.3\textwidth}
			\centering
			\includegraphics[scale=0.22]{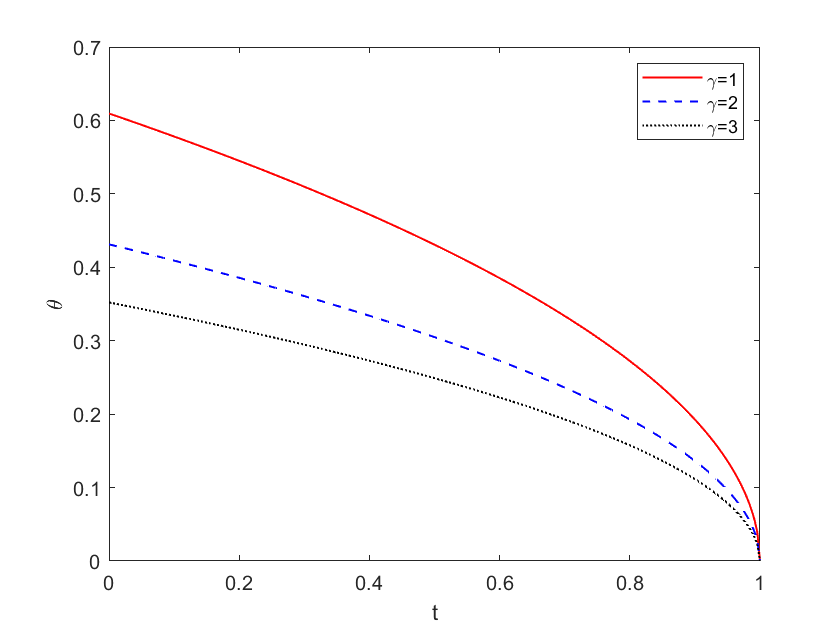}
            \caption{Effect of risk aversion}\label{b}
		\end{minipage}
        \begin{minipage}{0.3\textwidth}
			\centering
			\includegraphics[scale=0.22]{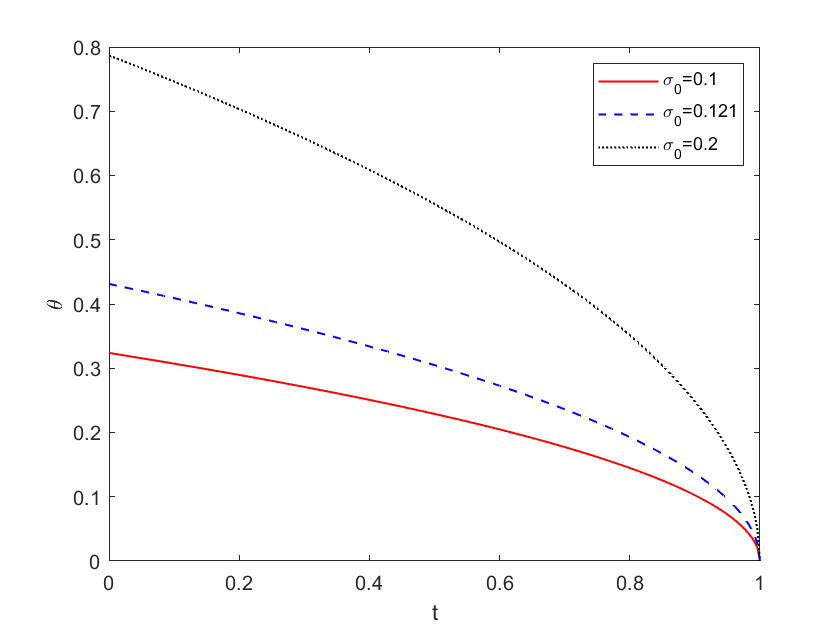}
            \caption{Effect of prior variance}\label{c}
		\end{minipage}
			
	\end{figure}

    When considering the impact of asset price volatility $\sigma$ and investment horizon $T$, the situation becomes slightly more complex. Figure \ref{figure_3} shows that there is no monotonic relationship between $\vartheta^*$ and $\sigma$. When $\sigma\rightarrow 0$, we have $\vartheta^*\rightarrow 0$, because the observed asset price process $$\frac{\d S_t}{S_t}=\mu \d t+\sigma\d W_t,$$ 
    already contains sufficiently precise information about $\mu$, rendering the purchase of additional private signal unnecessary. To better interpret the behavior for larger $\sigma$, 
    we further normalize our model: $$\frac{\d S_t}{\sigma S_t}=\frac{\mu}{\sigma}\d t+\d W_t,$$
    $$\d \widetilde{B}_t=\frac{\mu}{\sigma}(\sigma\vartheta_t)\d t+\d B_t.$$
    Rather than the drift parameter $\mu$, consider the term $\frac{\mu}{\sigma}$ to be the source of model uncertainty. In this case, information obtained from asset price is of constant quality (the volatility is normalized to 1), and two counteracting forces determine the optimal level of $\vartheta^*$:
    \begin{itemize}
        \item A larger $\sigma$ reduces the prior variance of $\frac{\mu}{\sigma}$, thus lowering the perceived benefit of acquiring additional information.
        \item  However, a larger $\sigma$ also reduces the cost of acquiring private information, because obtaining private signal on $\frac{\mu}{\sigma}$ with instantaneous variance $\frac{1}{\vartheta_t^2\d t}$ takes merely $k(\frac{1}{\sigma^2}\vartheta_t^2)$.
    \end{itemize}  
    When $\sigma$ is small, the second effect dominates, leading to an increase in $\vartheta_0^*$ with $\sigma$. Beyond a critical point (approximately $\sigma\approx 0.12$), the first effect dominates and $\vartheta_0^*$ starts to decline.

    \begin{figure}[htbp]
		\centering
        \begin{minipage}{0.49\textwidth}
			\centering
			\includegraphics[scale=0.3]{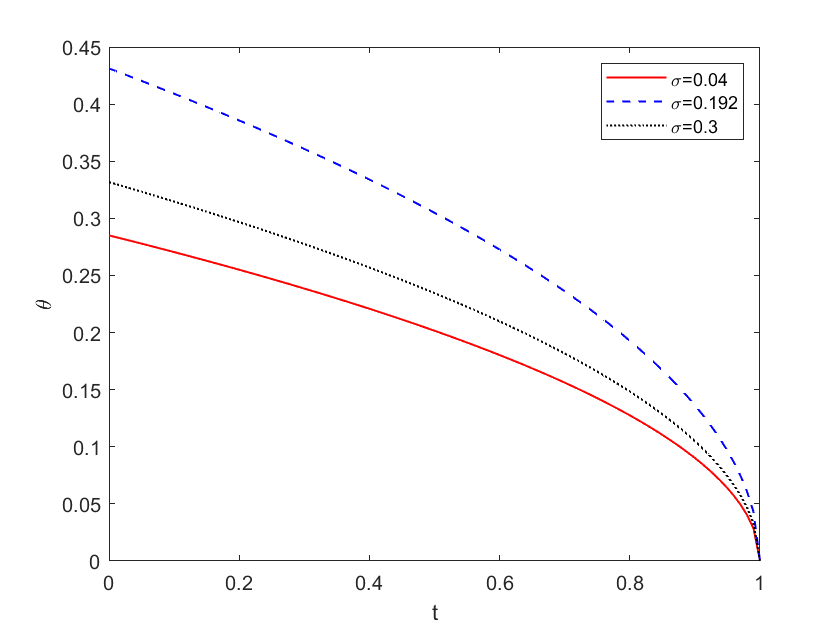}
		\end{minipage}\hfill
	\begin{minipage}{0.49\textwidth}
			\centering
			\includegraphics[scale=0.3]{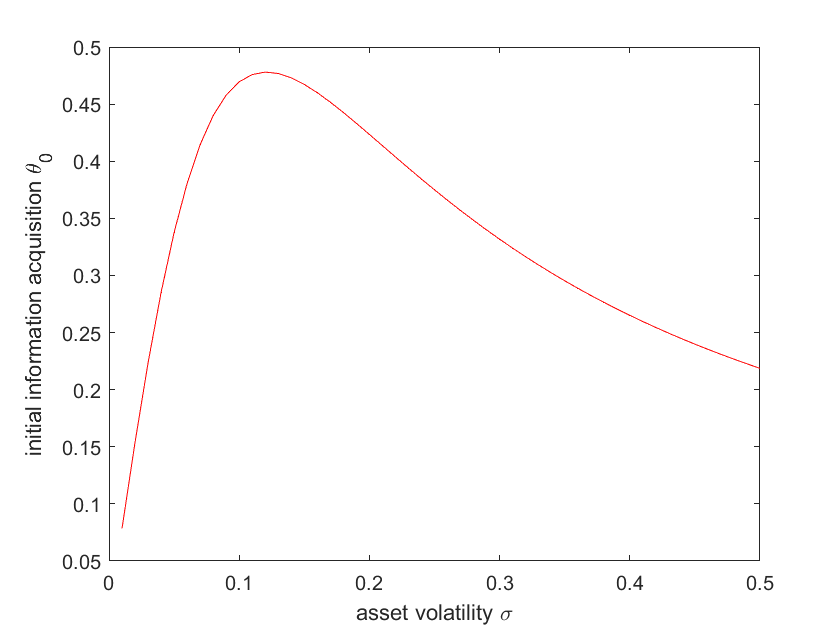}
		\end{minipage}\hfill
        \caption{Reaction of initial information acquisition to the change of asset volatility}\label{figure_3}
	\end{figure}

    \begin{figure}
		\centering
        \begin{minipage}{0.49\textwidth}
			\centering
			\includegraphics[scale=0.3]{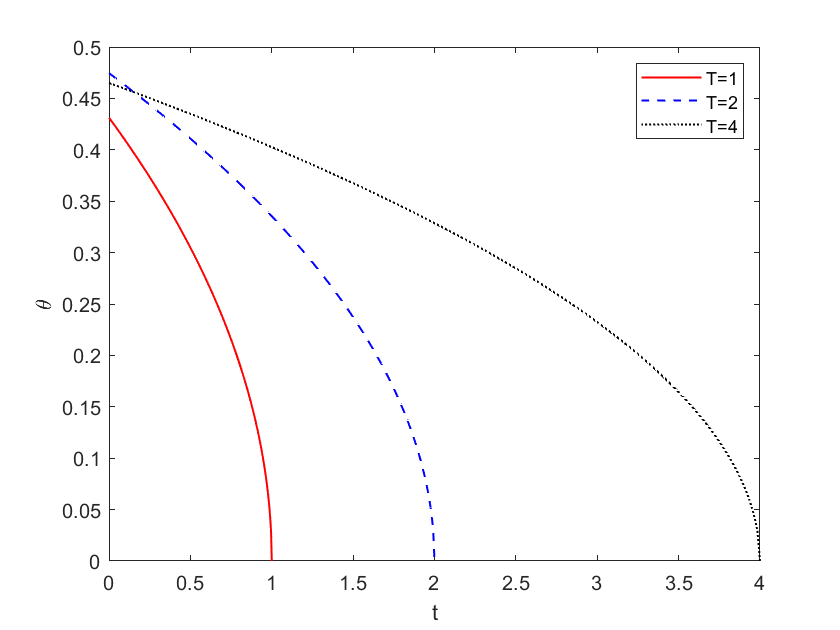}
		\end{minipage}
	\begin{minipage}{0.49\textwidth}
			\centering
			\includegraphics[scale=0.3]{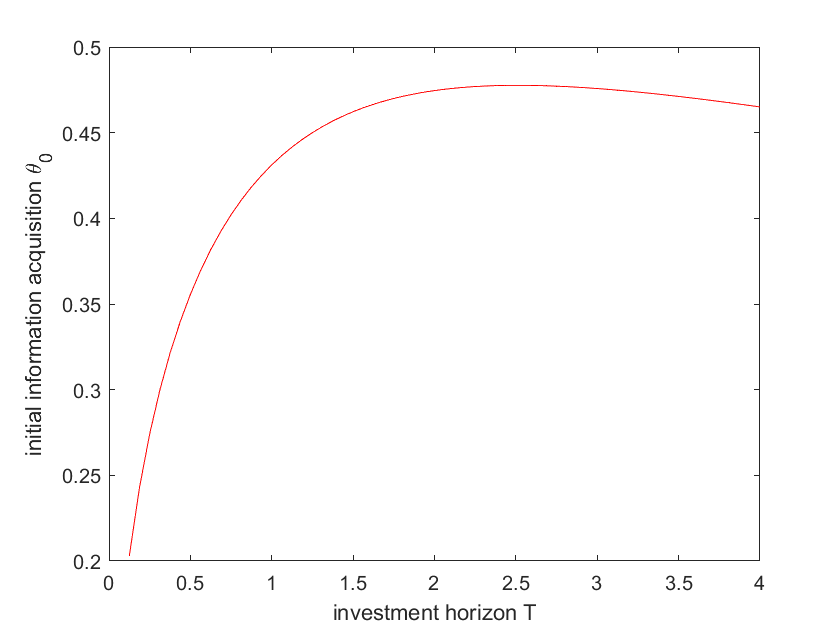}
		\end{minipage}
        \caption{Reaction of initial information acquisition to the change of investment horizon}\label{figure_2}
	\end{figure}

    Finally, we examine the impact of investment horizon, as illustrated in Figure \ref{figure_2}.  When $T\rightarrow 0$, we have $\vartheta^*\rightarrow 0$, because the investor has insufficient time to benefit from any learning. As the investment horizon lengthens, the value of early learning rises, and the investor finds it worthwhile to invest in information acquisition. Due to the increasing marginal cost (convexity of $k$), starting from a very high information acquisition level is not so profitable as starting from a moderate one and maintaining it over a longer period. As a result, the increasing trend in $\vartheta_0^*$ reverses at $T\approx 2.5$.

    \section{Discussions}\label{sec7}
    In this final section, we present a few extensions of the results from Section \ref{general}.
    \subsection{Economic interpretation of the Hamilton-Jacobi equation} \label{eco_int}
        
    
    
    Fix a deterministic information acquisition strategy $\vartheta$ and define $$Z_t=\frac{t}{\sigma^2}+\int_0^t \vartheta_s^2\d s,\qquad u_t=Z_t+\frac{1}{\sigma^2}(T-t).$$
    We consider the following stochastic control problem with $\varpi$ as the only controlled variable.
    
    In Remark \ref{primal_controlled_system_2}, we have mentioned that for the Gaussian prior distribution case, the optimal stochastic control system under $\P$ can be described in terms of the state processes $\{X_t\}_{0\leq t\leq T}$ and $\{\mu_t\}_{0\leq t\leq T}$, which is given by 
    \begin{align*}        \nonumber 
        &V^{0}(t,x,\mu)=\sup_{\varpi}\E\left[U({X}_T)| {X}_t=x, \mu_t=\mu\right]\\
       & \text{subject to }
        \begin{cases}
        \d {X}_t = \varpi_t[\mu_t\d t+\sigma \d M_t]-k(\vartheta_t^2)\d t,\\	
        \d\mu_t=\frac{\sigma_0^2}{Z_t\sigma_0^2+1}\left[\frac{1}{\sigma}\d M_t+\vartheta_t\d N_t\right].
        \end{cases}
        \end{align*}
    With the choice of CARA utility function $U(x)=-\frac{1}{\gamma}e^{-\gamma x}$, the above control problem can be reformulated as
    \begin{align*}
        \nonumber& V^{1}(t,x,\mu)\triangleq \frac{V^{0}(t,x,\mu)}{\exp\{\gamma\int_t^Tk(\vartheta_s^2)\d s\}}=\sup_{\varpi}\ \E\left[U(\widetilde{X}_T)| \widetilde{X}_t=x, \mu_t=\mu\right]\\
        &\text{subject to}
        \begin{cases}
        \d \widetilde{X}_t = \varpi_t[\mu_t\d t+\sigma \d M_t],\\	
        \d\mu_t=\frac{\sigma_0^2}{Z_t\sigma_0^2+1}\left[\frac{1}{\sigma}\d M_t+\vartheta_t\d N_t\right].
        \end{cases}
        \end{align*}
    The corresponding HJB equation is given by 
    \begin{equation}\label{71}
       \;\;\; \begin{cases}V^1_t+\;\underset{\pi\in\mathbb{R}}{\sup}\; \left\{V^1_x\pi\mu+\frac{1}{2}V^1_{xx}\pi^2\sigma^2+\frac{1}{2}V^1_{\mu\mu}\frac{\sigma_0^4}{(Z_t\sigma_0^2+1)^2}(\frac{1}{\sigma^2}+\vartheta_t^2)+V^1_{x\mu}\frac{\sigma_0^2\pi}{Z_t\sigma_0^2+1}\right\}=0,\\
        V^1(T,x,\mu)=U(x)=-\frac{1}{\gamma}e^{-\gamma x}.
        \end{cases}
    \end{equation}
    Motivated by the exponential-quadratic structure observed in previous sections, we make the following ansatz:
    $$V^1(t,x,\mu)=-\frac{1}{\gamma}e^{-\gamma x}\exp\left\{A_0(t)+A_1(t)\mu+\frac{1}{2}A_2(t)\mu^2\right\}.$$
    Substituting this into the HJB equation (\ref{71}) leads to the following system of ODEs:
    \begin{equation*}
        \begin{cases}
        A_0'(t)+\frac{1}{2}A_2(t)\frac{\sigma_0^4}{(Z_t\sigma_0^2+1)^2}(\frac{1}{\sigma^2}+\vartheta_t^2)+\frac{1}{2}A_1^2(t)\vartheta_t^2\frac{\sigma_0^4}{(Z_t\sigma_0^2+1)^2}=0,\\
        A_1'(t)+A_1(t)A_2(t)\vartheta_t^2\frac{\sigma_0^4}{(Z_t\sigma_0^2+1)^2}-\frac{A_1(t)}{\sigma^2}\frac{\sigma_0^2}{Z_t\sigma_0^2+1}=0,\\
        A_2'(t)+A_2^2(t)\vartheta_t^2\frac{\sigma_0^4}{(Z_t\sigma_0^2+1)^2}-\frac{1}{\sigma^2}\left(1+2A_2(t)\frac{\sigma_0^2}{Z_t\sigma_0^2+1}\right)=0,\\
        A_0(T)=A_1(T)=A_2(T)=0.
        \end{cases}
    \end{equation*}
    This system admits a unique solution given by
    $$A_0(t)=\frac{1}{2}\log \frac{\sigma_0^2Z_t+1}{\sigma_0^2u_t+1}+\int_t^T\frac{\sigma_0^2}{2\sigma^2}\frac{1}{\sigma_0^2u_s+1}\d s,\quad A_1(t)=0,\quad A_2(t)=\frac{t-T}{\sigma^2}\cdot\frac{\sigma_0^2Z_t+1}{\sigma_0^2u_t+1}.$$
    Therefore, the optimal response $\varpi(\vartheta)$ is expressed in feedback form \begin{equation}\label{optimal_trading}
        \varpi^*_t=\frac{\sigma_0^2Z_t+1}{\sigma_0^2u_t+1}\cdot \frac{\mu_t}{\sigma^2\gamma},
    \end{equation}  
    which coincides with Eq.\eqref{43} derived previously.
    
    The corresponding optimal value function $V(\vartheta)\triangleq V^0(0,x,\mu_0)$ is given by 
    $$V(\vartheta)=-\frac{1}{\gamma}e^{-\gamma x}\exp\left\{\gamma\int_0^Tk(\vartheta_t^2)\d t+\frac{1}{2}\log \left(\frac{\sigma^2}{\sigma_0^2 T+\sigma^2}\right)-\frac{\mu_0^2 T}{\sigma_0^2T+\sigma^2}+\int_0^T\frac{\sigma_0^2}{2\sigma^2}\frac{1}{\sigma_0^2u_t+1}\d t\right\}.$$
    By adding and subtracting constants that do not depend on the control, it is obvious that maximizing $V(\vartheta)$ amounts to solving the following deterministic control problem: 
    \begin{align}\label{inner_problem}
        \nonumber & \mini_{\vartheta}\;\left\{ \int_0^T\gamma k(\vartheta_t^2)\d t+\int_0^T\frac{\sigma_0^2}{2\sigma^2}\frac{1}{\sigma_0^2u_t+1}\d t-\frac{\mu_0^2}{2\sigma_0^2}\right\}\\
        &\text{subject to }
         \ \d u_t=\vartheta_t^2\d t,\;u_0=\frac{T}{\sigma^2}.
    \end{align} 
    Therefore, Eq.\eqref{29} can be interpreted as the dynamic programming equation associated with the maximization of $V(\vartheta)$ over all possible deterministic information acquisition strategies. In fact, the above derivation suggests another separation principle that the problems of optimal trading and active learning can be effectively decoupled, even though the optimization is conducted simultaneously. The investor first selects an information acquisition strategy by solving Problem \eqref{inner_problem}, which is only related (through $Z_t$) to the posterior variance of her drift estimate, then fixes this strategy to construct a portfolio proportional to the posterior mean, as given in \eqref{optimal_trading}.
    
    \subsection{Correlated Brownian motions}
    
    Consider the setting in which the Brownian motions driving the asset price process and the signal process are not mutually independent. To save the notation $\widetilde{B}$ for later use, the signal process is denoted by $\Xi=\{\Xi_t\}_{ 0\le t\le T}$ throughout this subsection. 
    
    An admissible information acquisition strategy $\theta:[0,T]\times C([0,T])\times C([0,T])\rightarrow \mathbb{R}$ is a bounded, progressively measurable functional such that the SDE system \begin{equation*}
        \begin{cases}
        \d S_t=S_t[\mu\d t+\sigma\d W_t],\\
        \d \Xi_t=\mu\theta(t,S,\Xi)\d t+\rho\d W_t+\sqrt{1-\rho^2}\d B_t,
        \end{cases}
    \end{equation*}
    admits a unique $\F^{\mu,W,B}$-progressively measurable solution, where $W$ and $B$ are independent Brownian motions, and $\rho\in(-1,1)$ is the correlation factor. 
    
    Let $ \{\vartheta_t=\theta(t,S,\Xi) \}_{0\leq t\leq T}$ be the realized process of admissible information acquisition strategy $\theta$. Define $\widetilde{W}$  and  $\widetilde{B}$ by
    $$\widetilde{W}_t=\frac{\mu}{\sigma}t+W_t,\qquad \widetilde{B}_t=\int_0^t\frac{\mu}{\sqrt{1-\rho^2}}(\vartheta_s-\frac{\rho}{\sigma})\d s+B_t,$$
    and the exponential process $\Psi$ by $$\Psi_t=\exp\left\{-\frac{\mu}{\sigma}W_t-\frac{\mu^2}{2\sigma^2}t-\int_0^t \frac{\mu}{\sqrt{1-\rho^2}}(\vartheta_s-\frac{\rho}{\sigma})\d B_s-\int_0^t \frac{\mu^2}{2(1-\rho^2)}(\vartheta_s-\frac{\rho}{\sigma})^2\d s\right\}.$$
    Under the equivalent measure $\widetilde{\P}$ defined by \eqref{measure_trans}, $\widetilde{W}$ and $\widetilde{B}$ are again independent Brownian motions with $\F^{S,\Xi}=\F^{\widetilde{W},\widetilde{B}}$. 
    
    Consider the following weak form stochastic control problem:
    \begin{align}\label{60}
    V(t,x,y,z)=\sup_{({\vartheta},\varpi)\in\widetilde{\mathcal{A}}}& \E^{\widetilde{\P}}\left[F({Y}_T,{Z}_T)U({X}_T)| {X}_t=x, {Y}_t=y, {Z}_t=z\right]\\ \nonumber
    \text{subject to }&
    \begin{cases}
     \d {X}_t = {\varpi}_t\sigma \d \widetilde{W}_t- k\left( \vartheta_t^2\right) \d t,\\	
    \d {Y}_t=\frac{1}{\sigma}\d \widetilde{W}_t+\frac{1}{\sqrt{1-\rho^2}}(\vartheta_t-\frac{\rho}{\sigma})\d \widetilde{B}_t,\\
    \d {Z}_t=(\frac{1}{\sigma^2}+\frac{1}{{1-\rho^2}}(\vartheta_t-\frac{\rho}{\sigma})^2)\d t, 
    \end{cases}
    \end{align} 
    where $\widetilde{\mathcal{A}}$ is the set of realized admissible strategies. The information cost function $k:[0,+\infty)\rightarrow [0,+\infty)$ is assumed to be a $C^2$ increasing function satisfying $k(0)=0$, $k''(x)>0,\forall x>0$ and $\underset{x\rightarrow+\infty}{\lim}k'(x)=+\infty$. Note that we do not impose condition on $k'(0)$ in this case.
    
    By introducing the auxiliary probability space, a strong form stochastic control system can be derived. The associated HJB equation is 
    \begin{equation}
\label{75}    \begin{cases}
        \underset{\pi,\theta}{\sup}\left\{V_t-k(\theta^2)V_x+(\frac{1}{\sigma^2}+\frac{1}{1-\rho^2}(\theta-\frac{\rho}{\sigma})^2)(V_z+\frac{1}{2}V_{yy})+\frac{1}{2}\pi^2\sigma^2V_{xx}+\pi V_{xy}\right\}=0,\\
        V(T,x,y,z)=F(y,z)U(x).\end{cases}
        \end{equation}
    By defining
    \begin{equation}\label{conjugate_2}
        k^*(x)\triangleq \underset{\theta\in\mathbb{R}}{\sup}\left\{\frac{1}{1-\rho^2}(\theta-\frac{\rho}{\sigma})^2x-k(\theta^2)\right\},\qquad x\ge 0,
    \end{equation} 
    equation \eqref{75} can be transformed into \eqref{conjugate}, with $\left(\frac{V_z+\frac{1}{2}V_{yy}}{V_x}\right)^+$ replaced by $\frac{V_z+\frac{1}{2}V_{yy}}{V_x}$.
    
    Assume now that $\rho\in (0,1)$, the case $\rho\in (-1,0)$ is analogous. It is clear that $0<k^*(x)<+\infty$ holds for any $ x>0$. Define the optimizer $$\theta^*(x)=\underset{\theta}{\mathrm{argmax}}\left\{\frac{1}{1-\rho^2}(\theta-\frac{\rho}{\sigma})^2x-k(\theta^2)\right\}.$$
    Then $\theta^*(0)=0$, and for each $x>0$, $\theta^*(x)$ is a non-empty subset of $(-\infty,0)$.
    In fact, for any given $x_0>0$ and $\theta\in\theta^*(x_0)$, we have the first order condition \begin{equation}\label{FOC}
        \frac{1}{1-\rho^2}(\theta-\frac{\rho}{\sigma})x_0-\theta k'(\theta^2)=0.
    \end{equation}
   Because the left-hand side of \eqref{FOC} is strictly increasing in $\theta\in (-\infty,0)$, and the right-hand side is strictly decreasing, we find that $\theta^*(x_0)$ consists of a singleton, which is the unique negative solution to Eq.\eqref{FOC}. By the implicit function theorem, $\theta^*\in C^1(0,+\infty)\cap C[0,+\infty)$, with derivative $$\frac{\d\theta^*(x)}{\d x}=-\frac{\frac{1}{1-\rho^2}(\theta^*(x)-\frac{\rho}{\sigma})}{\frac{x}{1-\rho^2}-k'(\theta^*(x)^2)-2\theta^*(x)^2k''(\theta^*(x)^2)}<0.$$
    Applying the envelop theorem yields
    $$(k^*)'(x)=\frac{1}{1-\rho^2}(\theta^*(x)-\frac{\rho}{\sigma})^2,$$ 
    which implies that $k^*\in C^2(0,+\infty)\cap C[0,+\infty)$ and is strictly increasing, convex, and positive.

    Define an auxiliary information acquisition process by $$\widetilde{\vartheta}_t=\frac{1}{\sqrt{1-\rho^2}}(\vartheta_t-\frac{\rho}{\sigma}),\quad t\in [0,T].$$
    Then we may now follow similar arguments as in Sections \ref{general} and \ref{eco_int} to establish the following result:

    \begin{theorem}
        Let the information cost $k$ be a $C^2$ increasing function satisfying $k(0)=0$, $k''(x)>0,\forall x>0$ and $\underset{x\rightarrow+\infty}{\lim}k'(x)=+\infty$. Let $k^*$ be defined by \eqref{conjugate_2}, and $\widetilde{k}^*(x)=-\gamma k^*(-\frac{x}{\gamma}),\;x\le0.$

        \begin{itemize}
            \item [(a)] Given a bounded deterministic information acquisition strategy $\vartheta$. In order to maximize her expected utility of terminal wealth, it is optimal for the investor to choose the trading strategy \begin{equation}\label{70}
                \varpi_t=\frac{\sigma_0^2 Z_t+1}{\sigma_0^2 u_t+1}\cdot\frac{\mu_t}{\sigma^2\gamma},
            \end{equation}
            where $\mu_t=\E[\mu|\F^{\widetilde{W},\widetilde{B}}_t]$ is the posterior mean and $$Z_t=\frac{t}{\sigma^2}+\int_0^t \widetilde{\vartheta}_s^2\d s,\qquad u_t=Z_t+\frac{1}{\sigma^2}(T-t).$$
            
            \item [(b)] If in addition $k'(0)\triangleq \underset{x\rightarrow 0+}{\lim}k'(x)\neq 0$, then the Hamilton-Jacobi equation \eqref{29} admits a unique classical solution $\Gamma(\cdot,\cdot)\in C^2([0,T)\times[0,+\infty))\cap C([0,T]\times[0,+\infty))$, which satisfies $\Gamma_u(t,u)\le 0$ for all $(t,u)\in[0,T)\times[0,+\infty)$. In this case the optimal information acquisition strategy is given by 
            $$\vartheta^*_t=\theta^*\left(-\frac{1}{\gamma}\Gamma_u(t,u_t)\right),$$
            where $u$ is the solution to ODE $$
            \d u_t=\frac{1}{1-\rho^2}\left(\theta^*\left(-\frac{1}{\gamma}\Gamma_u(t,u_t)\right)-\frac{\rho}{\sigma}\right)^2\d t, \qquad u(0)=\frac{T}{\sigma^2},$$
            and the optimal trading strategy is given by \eqref{70}.
            
            \item [(c)] If $k'(0)=0$, consider the regularized cost function $$k^{(b)}(x)=k(x)+bx,\quad b>0.$$
            Let $\vartheta^{(b)}$ be the optimal information acquisition strategy for cost function $k^{(b)}$, then the limiting function $$\vartheta^{(0)}_t=\underset{b\rightarrow 0}{\lim}\vartheta_t^{(b)}$$ is a well-defined continuous function on $ [0,T]$. Let $\varpi^{(0)}$ be defined by \eqref{70}, then $(\vartheta^{(0)},\varpi^{(0)})$ is the optimal realized information-and-trading strategy for cost function $k$.
            \end{itemize}
    \end{theorem}
    \subsection{CRRA utility function}
    If we consider the constant relative risk aversion (CRRA) utility function, given by $U(x)=\frac{1}{\gamma}x^{\gamma},\;\gamma\in (-\infty,0)$, then the variable separating structure of the problem fails. To apply the theoretical framework developed in Section \ref{general}, several modifications to the model are required.
    
    First, the controlled variable $\varpi$ must be reinterpreted as the proportion of wealth invested in the risky asset, rather than the monetary amount. Second, the cost of information acquisition should be modeled as being proportional to the investor’s current wealth. Although the second assumption may appear somewhat counterintuitive, it aligns with a core modeling choice in the seminal work by \citet{Andrei2020}.

   Under these modifications, the wealth process becomes
        $$\d X_t = \varpi_tX_t\sigma \d \widetilde{W}_t- k\left( \vartheta_t^2\right)X_t \d t,$$
    and the corresponding HJB equation becomes
        \begin{equation}\label{CRRA_HJB}
            \begin{cases}
        \underset{\pi,\theta}{\sup}\left\{V_t-k(\theta^2)xV_x+(\frac{1}{\sigma^2}+\theta^2)(V_z+\frac{1}{2}V_{yy})+\frac{1}{2}\pi^2x^2\sigma^2V_{xx}+\pi x V_{xy}\right\}=0,\\ V(T,x,y,z)=F(y,z)U(x).
        \end{cases}
        \end{equation}
        Making the following ansatz 
        $$V(t,x,y,z)=\frac{1}{\gamma}x^{\gamma}\exp\left\{P(t,z)+\frac{\mu_0y+\frac{1}{2}\sigma_0^2y^2}{\sigma_0^2z+\frac{\sigma_0^2\gamma}{\sigma^2(\gamma-1)}(T-t)+1}\right\}$$ and the transformed function $\Gamma$ by
        $$\Gamma(t,u)=P(t,u-\frac{\gamma(T-t)}{\sigma^2(\gamma-1)})+\frac{1}{2}\log(\sigma_0^2u+1)-\frac{\mu_0^2}{2\sigma_0^2}\frac{1}{\sigma_0^2u+1},$$
        we obtain the equation for $\Gamma$:\begin{equation*}\begin{cases}
        \Gamma_t(t,u)+\frac{\sigma_0^2\gamma}{2\sigma^2(\gamma-1)}\frac{1}{\sigma_0^2u+1}+\widetilde{k}^*(\Gamma_u(t,u))=0,\\
        \Gamma(T,u)=-\frac{\mu_0^2}{2\sigma_0^2},
    \end{cases}\end{equation*} 
    where $k^*$ is the Legendre-Fenchel transform of $k$ :
    $$\widetilde{k}^*(x)\triangleq \gamma k^*(\frac{1}{\gamma}x)-\frac{1}{\sigma^2(\gamma-1)}x.$$
        In this case, the essential properties of $\widetilde{k}^*$ remain and an analogue of Theorem \ref{general_characteristic} remains valid. In particular, the method of characteristics can still be employed to construct a classical solution $\Gamma(t,u)$ to the HJB equation \eqref{CRRA_HJB}.

    \begin{appendices}
    \section{Proof of Theorem \ref{value_function}}\label{TL_proof}
    In this appendix, we derive an explicit solution to the free-boundary problem \eqref{18}. We begin analyzing the following quasi-linear partial differential equation:
		\begin{align}\label{14}
			\begin{cases} h_t+a(h_z+\frac{1}{2}h_{yy})+\frac{b}{2}h_y^2+c=0,& (t,y,z)\in [0,T)\times \mathbb{R}\times (-\frac{1}{\sigma_0^2},+\infty),\\
				 h(T,y,z)=-\frac{1}{2}\log(z\sigma_0^2+1)+\frac{y^2\sigma_0^2+2y\mu_0-z\mu_0^2}{2(z\sigma_0^2+1)},& (y,z)\in\mathbb{R}\times (-\frac{1}{\sigma_0^2},+\infty),
                 \end{cases}
		\end{align}
        where $a$, $b$, and $c$ are fixed constants.
        \begin{lemma}\label{characteristic_method}
             Equation \eqref{14} admits a classical solution, given by
             $$h(t,y,z)=P(t,z)+\mu_0G(t,z;a-b) y+\frac{1}{2}\sigma_0^2G(t,z;a-b) y^2,$$ 
             where 
        $G(t,z;d)\triangleq\frac{1}{\sigma_0^2z+d\sigma_0^2(T-t)+1}$,
    \begin{align*}
     P(t,z)=\begin{cases}
    \frac{1}{2}\log G(t,z;a)-\frac{\mu_0^2G(t,z;a)}{2}(z+a(T-t))+c(T-t)&\\
    \quad +\frac{a}{2b}\log\frac{G(t,z;a-b)}{G(t,z;a)}+\frac{\mu_0^2}{2\sigma_0^2}(G(t,z;a-b)-G(t,z;a)),& b\neq 0,\\\frac{1}{2}\log G(t,z;a)-\frac{\mu_0^2G(t,z;a)}{2}(z+a(T-t))+c(T-t)&\\
    \quad +\frac{a\sigma_0^2G(t,z;a)}{2}(T-t),&b=0.
    \end{cases}
    \end{align*}
        \end{lemma}
	
 \begin{proof} We consider the solution of the quadratic form  $h(t,y,z)=P(t,z)+Q(t,z)y+\frac{1}{2}R(t,z)y^2$. Substituting into Eq.\eqref{14} and letting the coefficients of all powers of $y$ be equal to zero, we obtain the following first-order PDE system:
        \begin{align}\label{15}
		&\begin{cases}
		P_t(t,z)+aP_z(t,z)+\frac{a}{2}R(t,z)+\frac{b}{2}Q^2(t,z)+c=0,\\
		 P(T,z)=-\frac{1}{2}\log(z\sigma_0^2+1)-\frac{z\mu_0^2}{2(z\sigma_0^2+1)},
				\end{cases}\\
	\label{16}
			&\begin{cases}
			Q_t(t,z)+aQ_z(t,z)+bQ(t,z)R(t,z)=0,\\
			Q(T,z)=\frac{\mu_0}{z\sigma_0^2+1},
				\end{cases}\\
	\label{17}
			&\begin{cases}
			R_t(t,z)+aR_z(t,z)+bR^2(t,z)=0,\\
		 R(T,z)=\frac{\sigma_0^2}{z\sigma_0^2+1}.
				\end{cases}
		\end{align}
        These equations can be solved using the method of characteristics. For instance, the characteristic lines for Eq.\eqref{17} are given by $$z=a(t-T)+z_0,$$
        where $z_0$ is a constant to be determined. The equation restricted to a single characteristic line is reduced to an ODE: $$R^\prime(t)+bR^2=0,\;\;t\in [0,T);\qquad R(T)=\frac{\sigma_0^2}{z_0\sigma_0^2+1},$$ whose solution is $$R(t)=\frac{\sigma_0^2}{\sigma_0^2z_0-b\sigma_0^2(T-t)+1},\quad t\in[0,T].$$ 
        Substituting $z_0=z+a(T-t)$ gives the expression $R(t,z)=\sigma_0^2G(t,z;a-b)$. By using the same way, Eqs. \eqref{15} and \eqref{16} can be solved (notice that all three equations have the same characteristic lines).
        \end{proof}
		
        We now return to Eq.\eqref{18}. It is readily observed that $h_z+\frac{1}{2}h_{yy}+\frac{1}{2}h_{y}^2+c\gamma >0$ at $t=T$. Therefore, near the terminal time, we have
        \begin{align}\label{right_equation}
			\begin{cases}  h_t+\frac{1}{\sigma^2}(h_z+\frac{1}{2}h_{yy})=0,\\
				 h(T,y,z)=-\frac{1}{2}\log(z\sigma_0^2+1)+\frac{y^2\sigma_0^2+2y\mu_0-z\mu_0^2}{2(z\sigma_0^2+1)}.\end{cases}
		\end{align}
		By applying Lemma \ref{characteristic_method}, the solution to \eqref{right_equation} is given explicitly by
        \begin{align}\label{right_solution}
			\nonumber h(t,y,z)&=\frac{1}{2}\log H(t,z)-\frac{H(t,z)}{2}\left[\mu_0^2z+\frac{\mu_0^2-\sigma_0^2}{\sigma^2}(T-t)\right]\\&+H(t,z)\left[\mu_0y+\frac{1}{2}\sigma_0^2y^2\right].
		\end{align}
        The associated partial derivatives satisfy $$h_z+\frac{1}{2}h_{yy}+\frac{1}{2}h_{y}^2=
			 -\frac{\sigma_0^4}{2\sigma^2}H(t,z)^2(T-t).$$
        The solution evolves along the characteristic lines $z=\frac{1}{\sigma^2}(t-T)+z_0,\;(z_0>-\frac{1}{\sigma_0^2})$ until it hits the free boundary, which is determined by $$h_z+\frac{1}{2}h_{yy}+\frac{1}{2}h_{y}^2+c\gamma=0.$$
        Direct calculation yields the explicit form of the free boundary 
        \begin{equation*}
            z=-\frac{T-t}{\sigma^2}-\frac{1}{\sigma_0^2}+\sqrt{\frac{T-t}{2c\sigma^2\gamma}}\triangleq\delta(t).
        \end{equation*}
         Furthermore, because 
         \begin{equation*}
        \delta'(t)=\frac{1}{\sigma^2}-\frac{1}{4c\sigma^2\gamma}\left(\frac{T-t}{2c\sigma^2\gamma}\right)^{-1/2}<\frac{1}{\sigma^2},
        \end{equation*}
        we see that every point in the domain $\{z\ge \delta(t)\vee0\}\cap\{0\le t\le T\}$ is crossed by exactly one characteristic line of the form $z=\frac{1}{\sigma^2}(t-T)+z_0$. 

        We would expect \begin{align}\label{left_equation}
			\begin{cases}  h_t+(\frac{1}{\sigma^2}+\beta^2)(h_z+\frac{1}{2}h_{yy})+\frac{\beta^2}{2}h_{y}^2+\beta^2c\gamma=0,\\
				 h(t,y,\delta(t))=\frac{1}{2}\log H(t,\delta(t))-\frac{H(t,\delta(t))}{2}\left[\mu_0^2\delta(t)+\frac{\mu_0^2-\sigma_0^2}{\sigma^2}(T-t)\right]\\
			\qquad\qquad\qquad+H(t,\delta(t))\left[\mu_0y+\frac{1}{2}\sigma_0^2y^2\right],\end{cases}
		\end{align}
        as well as $h_z+\frac{1}{2}h_{yy}+\frac{1}{2}h_{y}^2+c\gamma \le0$ to hold in $\{0\le z<\delta(t)\}\cap\{0\le t<T\}$, if this region is not empty. Observe from Lemma \ref{characteristic_method} that the coefficients of the linear and quadratic terms in $y$ depend only on $a-b$, which remains unchanged between Eq.\eqref{right_equation} and Eq.\eqref{left_equation}. Therefore, we conjecture that the solution retains the form$$h(t,y,z)=P(t,z)+H(t,z)\left[\mu_0y+\frac{1}{2}\sigma_0^2y^2\right]$$ 
        in $\{0\le z<\delta(t)\}\cap\{0\le t<T\}$. Substituting this expression into Eq.\eqref{left_equation}, we obtain an equation that $P(t,z)$ should satisfy
        \begin{align}\label{19}
			\begin{cases}
			 P_t(t,z)+(\frac{1}{\sigma^2}+\beta^2)P_z(t,z)+\frac{1}{2}(\frac{1}{\sigma^2}+\beta^2)\sigma_0^2H(t,z)+\frac{\beta^2}{2}\mu_0^2H^2(t,z)+\beta^2c\gamma=0,\\
				 P(t,\delta(t))=\frac{1}{2}\log H(t,\delta(t))-\frac{H(t,\delta(t))}{2}\left[\mu_0^2\delta(t)+\frac{\mu_0^2-\sigma_0^2}{\sigma^2}(T-t)\right].\end{cases}
		\end{align}
    Let $(t_0,z_0)\in \{0\le z<\delta(t)\}\cap\{0\le t<T\}$ be a fixed point. The characteristic line for Eq.\eqref{19} is given by $z=(\frac{1}{\sigma^2}+\beta^2)(t-t_0)+z_0$. Denote by $(t_1,\delta(t_1))$ intersection point between this characteristic line and the free boundary, i.e., $t_1$ is the solution of 
    $$\left(\frac{1}{\sigma^2}+\beta^2\right)(t_1-t_0)+z_0=-\frac{T-t_1}{\sigma^2}-\frac{1}{\sigma_0^2}+\sqrt{\frac{T-t_1}{2c\sigma^2\gamma}}.$$ 
    Using the same method as in Lemma \ref{characteristic_method}, we obtain
        \begin{align}\label{sol_left}
    	\nonumber P(t_0,z_0)=\;&P(t_1,\delta(t_1))+\beta^2c\gamma(t_1-t_0)-\frac{\mu_0^2}{2\sigma_0^2}\left[H(t_1,\delta(t_1))-H(t_0,z_0)\right]\\
			&-\frac{\beta^2+\frac{1}{\sigma^2}}{2\beta^2}\left[\log H(t_1,\delta(t_1))-\log H(t_0,z_0)\right].
		\end{align}
        Using the fact that $\frac{\partial t_1}{\partial t_0}+(\frac{1}{\sigma^2}+\beta^2)\frac{\partial t_1}{\partial z_0}=0,$ we can readily verify that the function $P(t,z)$ defined above actually solves Eq.\eqref{left_equation} in $\{0\le z<\delta(t)\}\cap\{0\le t<T\}$. We prove next that $h_z+\frac{1}{2}h_{yy}+\frac{1}{2}h_{y}^2+c\gamma\le 0$ holds in this region. 
        
        Using the implicit function theorem, we have $$ \frac{\partial t_1}{\partial z_0}=-\frac{1}{\frac{1}{4c\sigma^2\gamma}(\frac{T-t_1}{2c\sigma^2\gamma})^{-1/2}+\beta^2}.$$ 
        which leads to the following important identity:
        \begin{equation}\label{1}
            \frac{\partial t_1}{\partial z_0}\left(-\delta'(t_1)+\beta^2+\frac{1}{\sigma^2}\right)+1=0.
        \end{equation} 
         Define $p(t)\triangleq P(t,\delta(t))$, then direct calculation from \eqref{19} gives \begin{align}
			p'(t_1)=-\frac{H(t_1,\delta(t_1))^2}{2}\cdot&\left[\sigma_0^4\delta(t_1)\delta'(t_1)+\frac{2\sigma_0^4}{\sigma^2}\delta'(t_1)(T-t_1)\right.\nonumber\\
            &+\left.\sigma_0^2\delta'(t_1)+\mu_0^2\delta'(t_1)-\frac{\sigma_0^4}{\sigma^4}(T-t_1)-\frac{\mu_0^2}{\sigma^2}\right].
		\end{align}
        Moreover, from the definition of the free boundary, we have \begin{equation}\label{3}
            c\gamma=\frac{1}{2\sigma^2}H(t_1,\delta(t_1))^2\sigma_0^4(T-t_1).
        \end{equation}
        Combining the Eqs. \eqref{sol_left}--\eqref{3}, with a little bit of algebra, we can verify that \begin{align}\label{boundary_derivative}
			&\nonumber \left.h_z+\frac{1}{2}h_{yy}+\frac{1}{2}h_y^2+c\gamma\right|_{(t_0,z_0)}\\
			&\nonumber =P_z(t_0,z_0)+H_z(t_0,z_0)\left[\mu_0y+\frac{1}{2}\sigma_0^2y^2\right]+\frac{\sigma_0^2}{2}H(t_0,z_0)+\frac{1}{2}H(t_0,z_0)^2\left[\mu_0+\sigma_0^2y\right]^2+c\gamma\\
			&\nonumber =p'(t_1)\frac{\partial t_1}{\partial z_0}+\beta^2 c\gamma\frac{\partial t_1}{\partial z_0}+\frac{\mu_0^2}{2\sigma_0^2}\left[\left(\sigma_0^2+\sigma_0^2\beta^2\frac{\partial t_1}{\partial z_0}\right){H(t_1,\delta(t_1))^2}-\sigma_0^2{H(t_0,z_0)^2}\right]\\
			&\nonumber \hphantom{H(t_0,z_0)}+\frac{\beta^2+\frac{1}{\sigma^2}}{2\beta^2}\left[\left(\sigma_0^2+\sigma_0^2\beta^2\frac{\partial t_1}{\partial z_0}\right){H(t_1,\delta(t_1))}-\sigma_0^2{H(t_0,z_0)}\right]-\mu_0\sigma_0^2y H(t_0,z_0)^2\\
			&\nonumber \hphantom{H(t_0,z_0)}-\frac{H(t_0,z_0)^2}{2}\sigma_0^4y^2+\frac{\sigma_0^2H(t_0,z_0)}{2}+\frac{H(t_0,z_0)^2}{2}(\mu_0+\sigma_0^2y)^2+c\gamma\\
			&\nonumber =\frac{H(t_1,\delta(t_1))^2}{2}\left(\frac{\partial t_1}{\partial z_0}\left(-\delta'(t_1)+\beta^2+\frac{1}{\sigma^2}\right)+1\right)\left(\sigma_0^2+\mu_0^2+\sigma_0^4\delta(t_1)+\frac{2\sigma_0^4}{\sigma^2}(T-t_1)\right)\\
			&\nonumber \hphantom{H(t_0,z_0)}+ \frac{\sigma_0^2}{2\beta^2\sigma^2}\left[H(t_1,\delta(t_1))-H(t_0,z_0)\right]\\
			&=\frac{\sigma_0^2}{2\beta^2\sigma^2}\left[H(t_1,\delta(t_1))-H(t_0,z_0)\right]\le 0.
		\end{align}
        Notice that we have also shown that $h_z+\frac{1}{2}h_{yy}+\frac{1}{2}h_y^2+c\gamma|_{(t_0,z_0)}\rightarrow 0$ as $(t_0,z_0)\rightarrow (t_1,\delta(t_1))$.

        It remains to show that the partial derivatives of $h(t,y,z)$ are continuous on the free boundary. To this end, we examine the directional derivatives. Take a point $(t_0,\delta(t_0))$ from the free boundary and another point $(t_0,z_0)\in \{0\le z<\delta(t)\}\cap\{0\le t<T\}$. As $z_0\rightarrow \delta(t_0)$, the implicit function theorem implies $$\frac{t_1-t_0}{\delta(t_0)-z_0}\rightarrow \frac{1}{\frac{1}{4c\sigma^2\gamma}(\frac{T-t_0}{2c\sigma^2\gamma})^{-1/2}+\beta^2},$$
        where $(t_1,\delta(t_1))$ is the point on the free boundary associated with $(t_0,z_0)$ via the characteristic line.
        Using this limiting behavior, we compute the left partial derivative:
        \begin{align*}
			P_{z-}(t_0,\delta(t_0))&\equiv\underset{z_0\rightarrow\delta(t_0)-}{\lim}\frac{P(t_0,\delta(t_0))-P(t_0,z_0)}{\delta(t_0)-z_0}\\
			&=\underset{z_0\rightarrow\delta(t_0)-}{\lim}\left\{\frac{P(t_0,\delta(t_0))-P(t_1,\delta(t_1))}{\delta(t_0)-z_0}-\beta^2c\gamma\frac{t_1-t_0}{\delta(t_0)-z_0}\right.\\
			& \hphantom{\lim}+\frac{\mu_0^2(t_1-t_0)}{2\sigma_0^2(\delta(t_0)-z_0)}\cdot\frac{H(t_1,\delta(t_1))-H(t_0,z_0)}{t_1-t_0}\\
			&\hphantom{\lim}+\left.\frac{(\beta^2+\frac{1}{\sigma^2})(t_1-t_0)}{2\beta^2(\delta(t_0)-z_0)}\cdot\frac{\log H(t_1,\delta(t_1))-\log H(t_0,z_0)}{t_1-t_0}\right\}\\
			& = -p'(t_0)\frac{1}{\frac{1}{4c\sigma^2\gamma}(\frac{T-t_0}{2c\sigma^2\gamma})^{-1/2}+\beta^2}-\frac{\beta^2c\gamma}{\frac{1}{4c\sigma^2\gamma}(\frac{T-t_0}{2c\sigma^2\gamma})^{-1/2}+\beta^2}\\
			&\hphantom{\lim}- \frac{\beta^2\sigma_0^2 }{\frac{1}{4c\sigma^2\gamma}(\frac{T-t_0}{2c\sigma^2\gamma})^{-1/2}+\beta^2}\left[\frac{\mu_0^2}{2\sigma_0^2}H(t_0,\delta(t_0))^2+\frac{\beta^2+\frac{1}{\sigma^2}}{2\beta^2}H(t_0,\delta(t_0))\right].
		\end{align*}
        Following the same procedure as in Eq.\eqref{boundary_derivative}, we can directly verify that \begin{equation}\label{left_derivative}
            h_{z-}+\frac{1}{2}h_{yy}+\frac{1}{2}h_y^2+c\gamma=0
        \end{equation} holds on the free boundary. 
        
        Using the explicit expression \eqref{right_solution}, we see that Eq.\eqref{left_derivative} holds with $h_{z-}$ replaced by $h_{z+}$. Therefore, the partial derivatives $h_z$ and $P_z$ are well-defined and continuous across the free boundary.
        Similarly, using \eqref{right_solution} and \eqref{sol_left} again, we see that for each $z_0>0$, the mapping $$t\mapsto P\left(t,z_0+(t-T)\left(\frac{1}{\sigma^2}+\beta^2\right)\right)$$ 
            has a well-defined and continuous derivative in the domain $\{t\in [0,T):z_0+(t-T)(\frac{1}{\sigma^2}+\beta^2)\ge 0\}$.
            Hence, $P$ has directional derivatives in two linearly independent directions, both of which are continuous. We conclude that $P\in C^1$, and consequently, $h\in C^{1,2,1}$.

\section{Proof of Theorem \ref{verification_theorem}}\label{verif}
        The admissibility of $(\theta^*,\pi^*)$ is immediate. Let $(\widehat{\vartheta}^*,\widehat{\varpi}^*)$ be the realization of  $(\theta^*,\pi^*)$ by $(\widehat W,\widehat B)$. By Theorem \ref{recovery}, it suffices to show that $(\widehat{\vartheta}^*,\widehat{\varpi}^*)$ constitutes an optimal solution of Problem \eqref{Q_value_function}. 
        
        For each $(\widehat{\vartheta},\widehat{\varpi})\in\widehat{\mathcal{A}}$, denote by $\{\widehat{X}_t\}$, $\{\widehat{Y}_t\}$,$\{\widehat{Z}_t\}$  the corresponding state processes with initial conditions $\widehat{X}_0=x, \widehat{Y}_0=0, \widehat{Z}_0=0$, respectively. Consider an exponential supermartingale $\zeta=\{\zeta_t,0\le t\le T\}$, initialized at $\zeta_0=1$ and satisfying the SDE:
        $$\d\zeta_t=\zeta_t\left[\left(\underbrace{\frac{1}{\sigma}\frac{\mu_0+\sigma_0^2 \widehat{Y}_t}{\sigma_0^2 \widehat{Z}_t+\frac{\sigma_0^2}{\sigma^2}(T-t)+1}-\gamma\sigma\widehat{\varpi}_t}_{\phi_t}\right)\d\widehat{W}_t+\underbrace{\widehat{\vartheta}_t\frac{\mu_0+\sigma_0^2 \widehat{Y}_t}{\sigma_0^2 \widehat{Z}_t+\frac{\sigma_0^2}{\sigma^2}(T-t)+1}}_{\varphi_t}\d\widehat{B}_t\right].$$
        It is clear that $\E^{\widehat{\P}}[\zeta_T]\le 1$. Let $V$ be the solution to the HJB equation provided in Corollary \ref{cor:V}. An application of Itô's formula yields
        \begin{align*}      &\d(V(t,\widehat{X}_t,\widehat{Y}_t,\widehat{Z}_t)\zeta_t^{-1})\\
            &=\zeta_t^{-1}\left[V_t-k(\widehat{\vartheta}_t^2)V_x+\left(\frac{1}{\sigma^2}+\widehat{\vartheta}_t^2\right)\left(V_z+\frac{1}{2}V_{yy}\right)+\frac{1}{2}\widehat{\varpi}_t^2\sigma^2V_{xx}+\widehat{\varpi}_t V_{xy}\right]\d t\le 0,
        \end{align*}
        where equality holds if and only if $\widehat{\varpi}=\widehat{\varpi}^*$ and $\widehat{\vartheta}=\widehat{\vartheta}^*$.
        
        Using the terminal condition and taking expectation under $\widehat{\P}$, we derive $$\E^{\widehat{\P}}[\zeta_T]V(0,x,0,0)\ge \E^{\widehat{\P}}[V(T,\widehat{X}_T,\widehat{Y}_T,\widehat{Z}_T)]=\E^{\widehat{\P}}[F(\widehat{Y}_T,\widehat{Z}_T)U(\widehat{X}_T)],$$
        where the equality holds if and only if $\widehat{\varpi}=\widehat{\varpi}^*$ and $\widehat{\vartheta}=\widehat{\vartheta}^*$. Therefore, to establish the optimality as well as the uniqueness of $(\widehat{\vartheta}^*,\widehat{\varpi}^*)$, it remains to prove that $\E^{\widehat{\P}}[\zeta_T]=1$ for all  $(\widehat{\vartheta},\widehat{\varpi})\in\widehat{\mathcal{A}}$. 
        
        To verify Novikov's condition, consider the expectation $$\E^{\widehat{\P}}\left[ \exp\left\{ \frac{1}{2}\int_{t_1}^{t_2}\left(\phi_t^2+\varphi_t^2\right)\d t\right\}\right],$$
        where $0\le t_1<t_2\le T$.
        Throughout this proof, $A>0$ is a generic constant whose value may vary from line to line, but is always independent of $t_1$ and $t_2$. Let $f_s=f(s,\widehat{W},\widehat{B})$ and $g_s=g(s,\widehat{W},\widehat{B})$, where $f$ and $g$ are given in the definition of $\widehat{\mathcal{A}}$. Using  $\widehat{Y}_t=\frac{1}{\sigma}\widehat{W}_t+\int_0^t\widehat{\vartheta}_s\d \widehat{B}_s$ and the boundedness of $\widehat{\vartheta}_t$, 
        the above expectation is dominated by 
        \begin{align*}
    E(t_1,t_2)\triangleq\E^{\widehat{\P}}&\left[ \exp\left\{\int_{t_1}^{t_2} A\left(1+\underset{0\le t\le T}{\sup}\widehat{W}^2_t+\underset{0\le t\le T}{\sup}\left(\int_0^t\widehat{\vartheta}_s\d \widehat{B}_s\right)^2\right.\right.\right.\\
              &+\left.\left.\left.\underset{0\le t\le T}{\sup}\left(\int_0^tf_s\d \widehat{W}_s\right)^2+\underset{0\le t\le T}{\sup}\left(\int_0^tg_s\d \widehat{B}_s\right)^2\right)\d t\right\}\right].
        \end{align*}
    Applying Hölder’s inequality yields
        \begin{align*}
             E(t_1,t_2)^4 &\le A \E^{\widehat{\P}} \left[ \exp\left\{A(t_2-t_1)\underset{0\le t\le T}{\sup}\widehat{W}^2_t\right\}\right]\times\E^{\widehat{\P}}\left[\exp\left\{A(t_2-t_1)\underset{0\le t\le T}{\sup}\left(\int_0^t\widehat{\vartheta}_s\d \widehat{B}_s\right)^2\right\}\right]\\&\times\E^{\widehat{\P}}\left[\exp\left\{A(t_2-t_1)\underset{0\le t\le T}{\sup}\left(\int_0^tf_s\d \widehat{W}_s\right)^2\right\}\right]
            \times\E^{\widehat{\P}}\left[\exp\left\{A(t_2-t_1)\underset{0\le t\le T}{\sup}\left(\int_0^tg_s\d \widehat{B}_s\right)^2\right\}\right].
        \end{align*}
        For the first term, using the reflection principle of Brownian motion, we have
        \begin{align*}
            \E^{\widehat{\P}}\left[\exp\left\{A(t_2-t_1)\underset{0\le t\le T}{\sup}\widehat{W}_t^2\right\}\right]
             &=1+\int_0^{+\infty}A(t_2-t_1)e^{A(t_2-t_1)x}\widehat{\P}\left(\underset{0\le t\le T}{\sup}\widehat{W}_t^2\ge x\right)\d x\\
            & \le 1+2\int_0^{+\infty}A(t_2-t_1)e^{A(t_2-t_1)x}\widehat{\P}\left(\underset{0\le t\le T}{\sup}\widehat{W}_t\ge \sqrt{x}\right)\d x\\
            &= 1+4\int_0^{+\infty}A(t_2-t_1)e^{A(t_2-t_1)x}\widehat{\P}\left(\widehat{W}_T\ge \sqrt{x}\right)\d x\\
            & =1+4\int_0^{+\infty}A(t_2-t_1)e^{A(t_2-t_1)x}\left[1-\Phi\left(\sqrt{\frac{x}{T}}\right)\right]\d x\\
            & \le 1+4\int_0^{+\infty}A(t_2-t_1)e^{A(t_2-t_1)x}\sqrt{\frac{T}{2\pi x}}e^{-x/2T}\d x<+\infty,
        \end{align*}
        provided that $t_2-t_1$ is sufficiently small (note that $A$ is independent of the choice of $t_1$ and $t_2$). The remaining terms, take  $\E^{\widehat{\P}}\left[\exp\left\{A(t_2-t_1)\underset{0\le t\le T}{\sup}\left(\int_0^t\widehat{\vartheta}_s\d \widehat{B}_s\right)^2\right\}\right]$ for an example, can be treated similarly. This is because the stochastic integral $\int_0^t\widehat{\vartheta}_s\d \widehat{B}_s$ can be viewed as a time-changed Brownian motion ${B}_{\int_0^t \widehat{\vartheta}_s^2\d s}'$. Here ${B'}$ is a $\widehat{\P}$-Brownian motion and $\underset{0\le t\le T}{\sup}(\int_0^t\widehat{\vartheta}_s\d \widehat{B}_s)^2$ is dominated by $\underset{0\le t\le CT}{\sup}({B}_t')^2$, where $C$ is the constant in the admissibility condition. Thus, based on \citet[Corollary 3.5.14]{Karatzas1991}, we have $\E^{\widehat{\P}}[\zeta_T]=1$.
        
        \section{Proof of Theorem \ref{general_characteristic}}\label{GC_proof}

        The regularity condition $k''(0)>0$ ensures that the convex conjugate ${k}^*$ can be extended to a $C^2$ function defined on the whole real line.
        For notational convenience, we will henceforth assume that ${k}^*\in C^2(\mathbb{R})$.

        To simplify the analysis, we consider the time reversed equation \begin{equation}\label{time_rev}\begin{cases}
        \widetilde{\Gamma}_t(t,u)-\frac{\sigma_0^2}{2\sigma^2}\frac{1}{\sigma_0^2u+1}-\widetilde{k}^*(\widetilde{\Gamma}_u(t,u))=0,\\
        \widetilde{\Gamma}(0,u)=-\frac{\mu_0^2}{2\sigma_0^2},
    \end{cases}\end{equation}
    in which the Hamiltonian $$H(t,x,p)=-\frac{\sigma_0^2}{2\sigma^2}\frac{1}{\sigma_0^2x+1}-\widetilde{k}^*(p)$$
    is twice continuously differentiable. We solve this equation via the classical method of characteristics. The corresponding characteristic system of Eq.\eqref{time_rev} is given by
    \begin{equation}\label{characteristic_system}
        \begin{cases}
         \widetilde{\Gamma}_t'(s)=0,\\
         \widetilde{\Gamma}_u'(s)=-\frac{\sigma_0^4}{2\sigma^2}\frac{1}{(\sigma_0^2 u(s)+1)^2},\\
         \widetilde{\Gamma}'(s)=-\widetilde{\Gamma}_u(s)(\widetilde{k}^*)'(\widetilde{\Gamma}_u(s))+\frac{\sigma_0^2}{2\sigma^2}\frac{1}{\sigma_0^2u(s)+1}+\widetilde{k}^*(\widetilde{\Gamma}_u(s)),\\
         t'(s)=1,\\
         u'(s)=-(\widetilde{k}^*)'(\widetilde{\Gamma}_u(s)),
         \end{cases}
    \end{equation}
    where $(t(s),u(s))$ is the characteristic curve and $\widetilde{\Gamma}(s)\triangleq \widetilde{\Gamma}(t(s),u(s))\;(\widetilde{\Gamma}_u(s),\widetilde{\Gamma}_t(s))$ represents the value (partial derivatives) of the solution $\widetilde{\Gamma}$ along the characteristic curve. 
    
    At initial time $s=0$, the characteristic curve starts from a boundary point $(t(0),u(0))=(0,u_0)$, where $u_0$ is a constant to be determined. Therefore, the initial conditions for the system  \eqref{characteristic_system} are
 $$\widetilde{\Gamma}_t(0)=\frac{\sigma_0^2}{2\sigma^2}\frac{1}{\sigma_0^2u_0+1},\qquad \widetilde{\Gamma}_u(0)=0,\qquad \widetilde{\Gamma}(0)=-\frac{\mu_0^2}{2\sigma^2},\qquad t(0)=0,\qquad u(0)=u_0.$$

    We now isolate the essential part of the characteristic system, known as Hamilton’s ODE:
    \begin{equation}\label{31}
        \begin{cases}
         \widetilde{\Gamma}_u'(s)=-\frac{\sigma_0^4}{2\sigma^2}\frac{1}{(\sigma_0^2 u(s)+1)^2},& \widetilde{\Gamma}_u(0)=0,\\
         u'(s)=-(\widetilde{k}^*)'(\widetilde{\Gamma}_u(s)),& u(0)=u_0.
         \end{cases}
    \end{equation}
    We now state the main result, which implies the validity of Theorem \ref{general_characteristic}.
    \begin{theorem}
        Under Assumption \ref{ass:k:convex} and the additional assumption $k''(0)>0$, for each $u_0>0$, the system of ODEs \eqref{31} admits a unique local solution $(\widetilde{\Gamma}_u^{u_0}(\cdot), u^{u_0}(\cdot))$ such that $$(\widetilde{\Gamma}_u(s),u(s))\in (-\infty,+\infty)\times \left(-\frac{1}{2\sigma_0^2},+\infty\right).$$ The intersection of the maximal interval of existence and the positive time line $[0,+\infty)$ is denoted by $[0,t_0(u_0))$ and abbreviated as $[0,t_0)$ when there is no confusion. Denote by $\widetilde{\Gamma}^{u_0}_t(\cdot)$ and $\widetilde{\Gamma}^{u_0}(\cdot)$ the corresponding solution to  the  system \eqref{characteristic_system} for $s\in [0,t_0)$.
        
    For each $(t,u)\in [0,T]\times [0,+\infty)$, there exists a unique $u_0\ge0$ such that $t<t_0(u_0)$ and $u=u^{u_0}(t)$. Furthermore, $u_0$ is a $C^1 -$ function of $(t,u)$. Define \begin{equation*}
            \widetilde{\Gamma}(t,u)\triangleq \widetilde{\Gamma}^{u_0}(t)=-\frac{\mu_0^2}{2\sigma^2}+\int_0^t\left\{ -\widetilde{\Gamma}_u^{u_0}(s)(\widetilde{k}^*)'(\widetilde{\Gamma}_u^{u_0}(s))+\frac{\sigma_0^2}{2\sigma^2}\frac{1}{\sigma_0^2u^{u_0}(s)+1}+\widetilde{k}^*(\widetilde{\Gamma}_u^{u_0}(s))\right\}\d s.
        \end{equation*}
        Then $\widetilde{\Gamma}(t,u)\in C^2((0,T]\times[0,+\infty))\cap C([0,T]\times[0,+\infty))$ is the unique $C^2$-classical solution to Eq.\eqref{time_rev}, which satisfies $\widetilde{\Gamma}_u(t,u)\le 0$ for all $(t,u)\in(0,T]\times[0,+\infty)$, respectively.
    \end{theorem}
    
        The proof is divided into several steps:

        \paragraph{Existence and uniqueness of local solution.}
        To determine the characteristic curves lying in the plane $(t,u)\in [0,T]\times (-\frac{1}{2\sigma_0^2},+\infty)$, we are only interested in the $s\ge 0$ part of the solution to the Hamilton's ODE. Observe that every solution to Eq.\eqref{31} must satisfy $\widetilde{\Gamma}_u(s)\le 0$ for all $s\ge 0$. Therefore, arbitrarily extending the definition of ${k}^*$ on $(-\infty,0)$ actually makes no difference. By the fact that $(\widetilde{k}^*)'\in C^1(\mathbb{R})$, the right-hand side of Eq.\eqref{31} is locally Lipschitz continuous in $(\widetilde{\Gamma}_u,u)\in (-\infty,+\infty)\times (-\frac{1}{2\sigma_0^2},+\infty)$. Therefore, the existence and uniqueness of local solution, as well as the well-definedness of the maximal interval of existence, follow from a simple application of \citet[Theorem 1.3.1]{Hale1980}, which is known as the Picard-Lindel\"of theorem. 

        \paragraph{Properties of $\widetilde{\Gamma}_u^{u_0}$ and $u^{u_0}$.} Fix $u_0\ge0$ and consider$$f(s)\triangleq\widetilde{k}^*(\widetilde{\Gamma}^{u_0}_u(s))+\frac{\sigma_0^2}{2\sigma^2}\left(\frac{1}{\sigma_0^2u(s)+1}-\frac{1}{\sigma_0^2u_0+1}\right).$$
        Direct calculation shows that $f(0)=f'(s)=0,\;\forall s\in (0,t_0)$. Therefore, we have 
        $$\frac{1}{\sigma_0^2u(s)+1}=-\frac{2\sigma^2}{\sigma_0^2}\widetilde{k}^*(\widetilde{\Gamma}_u^{u_0}(s))+\frac{1}{\sigma_0^2u_0+1}.$$ 
        Plugging this expression back into the first equation of the system \eqref{31}, we obtain an equation that $\widetilde{\Gamma}_u^{u_0}$ should satisfy
        \begin{equation}\label{32}
    \widetilde{\Gamma}_u'(s)=-2\sigma^2\left[\widetilde{k}^*(\widetilde{\Gamma}_u(s))-\frac{\sigma_0^2}{2\sigma^2}\frac{1}{\sigma_0^2u_0+1}\right]^2,\qquad s\in [0,t_0).
    \end{equation}
    Notice that the right-hand side of Eq.\eqref{32} is also locally Lipschitz  w.r.t. $\widetilde{\Gamma}_u\in (-\infty,+\infty)$, thus it admits a unique local solution. This local solution, restricted to $[0,t_0)$, must be equal to $\widetilde{\Gamma}_u^{u_0}$. Clearly, $u^{u_0}$ is determined by \begin{equation}\label{33}
        u^{u_0}(t)=u_0-\int_0^t (\widetilde{k}^*)'(\widetilde{\Gamma}_u^{u_0}(s))\d s,\qquad t\in [0,t_0).
    \end{equation}
    The continuation result to solutions of ODEs gives $$(\widetilde{\Gamma}_u(s),u(s))\rightarrow \partial((-\infty,+\infty)\times (-\frac{1}{2\sigma_0^2},+\infty)),\qquad s\rightarrow t_0.$$ 
    Both $\widetilde{\Gamma}_u^{u_0}(s)$ and $u^{u_0}(s)$ are decreasing in $s\in[0,t_0)$. Therefore, we have either $\widetilde{\Gamma}_u^{u_0}(s)\rightarrow -\infty$ or $u^{u_0}(s)\rightarrow -\frac{1}{2\sigma_0^2}$ when $s\rightarrow t_0$. The former is easily ruled out from the structure of \eqref{31}, and we conclude that $u^{u_0}(s)\xrightarrow{s\rightarrow t_0} -\frac{1}{2\sigma_0^2}$.
    
    \paragraph{Uniqueness of $u_0$.} We shall argue by contradiction. Suppose that two characteristic curves starting from $0\le u_1<u_2$ cross the same point $(t,u)\in [0,T]\times[0,+\infty)$, then clearly $t>0$. From  Eq.\eqref{32} we have, for $i=1,2$: $$(\widetilde{\Gamma}^{u_i}_u)'(s)=-2\sigma^2\left[\widetilde{k}^*(\widetilde{\Gamma}^{u_i}_u(s))-\frac{\sigma_0^2}{2\sigma^2}\frac{1}{\sigma_0^2u_i+1}\right]^2,\qquad \widetilde{\Gamma}^{u_i}_u(0)=0.$$ By the first comparison principle of ordinary differential equations, we have $\widetilde{\Gamma}_u^{u_1}(s)<\widetilde{\Gamma}_u^{u_2}(s)$ holds for every $s\in (0,t)$, which contradicts $$u_1-\int_0^t (\widetilde{k}^*)'(\widetilde{\Gamma}^{u_1}_u(s))\d s=u_2-\int_0^t (\widetilde{k}^*)'(\widetilde{\Gamma}^{u_2}_u(s))\d s,$$
    a consequence of \eqref{33}.

    \paragraph{Existence of $u_0$.} Consider an arbitrary but fixed point $(t,u)\in [0,T]\times [0,+\infty)$,  define $$M_1=\underset{x\in (-1,0)}{\sup}|t(\widetilde{k}^*)'(x)|=t(\widetilde{k}^*)'(-1),\qquad M_2=\inf\left\{x>0:\frac{\sigma_0^4}{2\sigma^2}\frac{1}{(\sigma_0^2 x+1)^2}<\frac{1}{t}\right\}.$$
    Then from the system \eqref{31}, it is easy to see that for any $\overline{u}>u+M_1+M_2$, we have $\widetilde{\Gamma}_u^{\overline{u}}(t)\ge -1$ and $u^{\overline{u}}(t)\ge u$. Consider 
    $$u_*=\inf\{u_0\ge0:t<t_0(u_0)\}.$$ 
    By the continuity of the solution with respect to the initial data, we have $t_0(u_*)\le t$.
    Therefore, there exists $t'<t$ such that $u^{u_*}(t')<0$, by the property of $u^{u_0}$ that we have just proved. Using the continuity of the solution with respect to the initial data again, we conclude that $u^{\underline{u}}(t')<0$ holds for some $\underline{u}>u_*$, which implies $u^{\underline{u}}(t)<0\le u$. 
    Finally, an application of the intermediate value theorem shows the existence of $u_0\in [\underline{u},\overline{u}]$ such that $u^{u_0}(t)=u$, as desired.

    \paragraph{Continuous differentiability of $u_0$ as a function of $(t,u)$.} The classical differentiability result of solution with respect to initial data shows that $\widetilde{\Gamma}_u^{u_0}(s)$ and $u^{u_0}(s)$ are $C^1$ in both $s$ and $u_0$. Furthermore, as $(\widetilde{k}^*)''\le0$ and $\frac{\partial \widetilde{\Gamma}_u^{u_0}(s)}{\partial u_0}\ge 0$, we have $$\frac{\partial u^{u_0}(t)}{\partial u_0}=1-\int_0^t(\widetilde{k}^*)''(\widetilde{\Gamma}_u^{u_0}(s))\frac{\partial \widetilde{\Gamma}_u^{u_0}(s)}{\partial u_0}\d s\ge 1>0.$$ 
    Consequently, the Jacobian matrix $$\frac{\partial (t,u^{u_0}(t))}{\partial (t,u_0)}=\begin{pmatrix} 1 & 0\\
    (u^{u_0})'(t) & \frac{\partial u^{u_0}(t)}{\partial u_0}\end{pmatrix}$$is invertible. As the mapping from $\{(t,u_0):u^{u_0}(t)\ge 0\}$ to $(t,u^{u_0}(t))\in [0,T]\times [0,+\infty)$ is bijective, we may apply the inverse mapping theorem and the result follows.

    \paragraph{$\widetilde{\Gamma}^{u_0(t,u)}(t)$ solves $\eqref{time_rev}$.}  From now on, we will regard $\widetilde{\Gamma}^{u_0}(s),\widetilde{\Gamma}_t^{u_0}(s)$ and $u^{u_0}(s)$ as $C^1$ functions of $(s,u_0)$. Fix a single characteristic curve starting from $u_0\ge0$, consider $$g(s)=\widetilde{\Gamma}_t^{u_0}(s)-\frac{\sigma_0^2}{2\sigma^2}\frac{1}{\sigma_0^2u^{u_0}(s)+1}-\widetilde{k}^*(\widetilde{\Gamma}_u^{u_0}(s)).$$ 
    Then we have $g(0)=0$ and $$g'(s)=(\widetilde{\Gamma}_t^{u_0})'(s)+\frac{\sigma_0^2}{2\sigma^2}\frac{\sigma_0^2(u^{u_0})'(s)}{(\sigma_0^2u^{u_0}(s)+1)^2}-(\widetilde{k}^*)'(\widetilde{\Gamma}_u^{u_0}(s))(\widetilde{\Gamma}_u^{u_0})'(s)=0.$$ 
    Therefore, on each characteristic curve we have $g(s)\equiv 0$ for $s\in [0,t_0)$. The proof is complete once we have verified that $\frac{\partial}{\partial t}\widetilde{\Gamma}(t,u)=\widetilde{\Gamma}_t^{u_0(t,u)}(t)$ and $\frac{\partial}{\partial u}\widetilde{\Gamma}(t,u)=\widetilde{\Gamma}_u^{u_0(t,u)}(t)$,
    and the continuous differentiability of $\widetilde{\Gamma}_t^{u_0}(s)$ and $\widetilde{\Gamma}_u^{u_0}(s)$ implies that $\widetilde{\Gamma}(t,u)\triangleq\widetilde{\Gamma}^{u_0(t,u)}(t)\in C^2((0,T]\times[0,+\infty))\cap C([0,T]\times[0,+\infty))$.
    
    In fact, from the characteristic system we have \begin{align}\label{35}
        \nonumber\frac{\partial}{\partial s}\widetilde{\Gamma}^{u_0}(s)&=-\widetilde{\Gamma}_u^{u_0}(s)(\widetilde{k}^*)'(\widetilde{\Gamma}_u^{u_0}(s))+\frac{\sigma_0^2}{2\sigma^2}\frac{1}{\sigma_0^2u^{u_0}(s)+1}+\widetilde{k}^*(\widetilde{\Gamma}_u^{u_0}(s))\\
        \nonumber&=-\widetilde{\Gamma}_u^{u_0}(s)(\widetilde{k}^*)'(\widetilde{\Gamma}_u^{u_0}(s))+\frac{\sigma_0^2}{2\sigma^2}\frac{1}{\sigma_0^2u_0+1}\\
        &=\widetilde{\Gamma}_u^{u_0}(s)\frac{\partial}{\partial s}u^{u_0}(s)+\widetilde{\Gamma}_t^{u_0}(s).
    \end{align}
    We show next that \begin{equation}\label{36}
        \frac{\partial}{\partial u_0}\widetilde{\Gamma}^{u_0}(s)=\widetilde{\Gamma}_u^{u_0}(s)\frac{\partial}{\partial u_0}u^{u_0}(s).
    \end{equation} For this we define $r(s)=\frac{\partial}{\partial u_0}\widetilde{\Gamma}^{u_0}(s)-\widetilde{\Gamma}_u^{u_0}(s)\frac{\partial}{\partial u_0}u^{u_0}(s),\;s\in [0,t_0(u_0))$, where $u_0\ge 0$ is fixed. It is clear that $r(0)=0$. To calculate its derivative, note that from the classical theory of ODEs, we have the following matrix linear variational equation:
    $$\frac{\partial}{\partial s}\begin{pmatrix}
        \frac{\partial}{\partial u_0}\widetilde{\Gamma}_u^{u_0}(s)\\
        \frac{\partial}{\partial u_0}u^{u_0}(s)
    \end{pmatrix}=\begin{pmatrix}
        0 &\frac{\sigma_0^6}{\sigma^2}\frac{1}{(\sigma^2u^{u_0}(s)+1)^3}\\
        -(\widetilde{k}^*)''(\widetilde{\Gamma}_u^{u_0}(s)) & 0
    \end{pmatrix}\begin{pmatrix}
        \frac{\partial}{\partial u_0}\widetilde{\Gamma}_u^{u_0}(s)\\
        \frac{\partial}{\partial u_0}u^{u_0}(s)
    \end{pmatrix}.$$
    This shows that we can  interchange the order of derivatives$$\frac{\partial}{\partial s}(\frac{\partial}{\partial u_0}\widetilde{\Gamma}_u^{u_0}(s))=\frac{\sigma_0^6}{\sigma^2}\frac{1}{(\sigma^2u^{u_0}(s)+1)^3}\frac{\partial}{\partial u_0}u^{u_0}(s)=\frac{\partial}{\partial u_0}(\frac{\partial}{\partial s}\widetilde{\Gamma}_u^{u_0}(s))$$
    even if $\widetilde{\Gamma}_u^{u_0}(s)$ is merely a $C^1$ function of $(s,u_0)$. The same property holds for $u^{u_0}(s)$ and $\widetilde{\Gamma}^{u_0}(s)$, and we finally obtain
    \begin{align*}
        r'(s)&=\frac{\partial^2}{\partial s\partial u_0}\widetilde{\Gamma}^{u_0}(s)-\widetilde{\Gamma}_u^{u_0}(s)\frac{\partial^2}{\partial s\partial u_0}u^{u_0}(s)-\frac{\partial}{\partial s}\widetilde{\Gamma}_u^{u_0}(s)\frac{\partial}{\partial u_0}u^{u_0}(s)\\
        &=\frac{\partial}{\partial u_0}\widetilde{\Gamma}_u^{u_0}(s)\frac{\partial}{\partial s}u^{u_0}(s)+\frac{\partial}{\partial u_0}\widetilde{\Gamma}_t^{u_0}(s)-\frac{\partial}{\partial s}\widetilde{\Gamma}_u^{u_0}(s)\frac{\partial}{\partial u_0}u^{u_0}(s)\\
        &=-\frac{\partial}{\partial u_0}\widetilde{\Gamma}_u^{u_0}(s)(\widetilde{k}^*)'(\widetilde{\Gamma}_u^{u_0}(s))+\frac{\partial}{\partial u_0}\widetilde{\Gamma}_t^{u_0}(s)+\frac{\sigma_0^4}{2\sigma^2}\frac{1}{(\sigma_0^2 u^{u_0}(s)+1)^2}\frac{\partial}{\partial u_0}u^{u_0}(s)\\
        &=0.
    \end{align*}
    Note that the second equality is an application of \eqref{35} and the last equality is obtained by taking derivative w.r.t. $u_0$ in the equation
    $$\widetilde{\Gamma}_t^{u_0}(s)-\frac{\sigma_0^2}{2\sigma^2}\frac{1}{\sigma_0^2u^{u_0}(s)+1}-\widetilde{k}^*(\widetilde{\Gamma}_u^{u_0}(s))=0,$$and \eqref{36} is now proved. With the help of \eqref{35}-\eqref{36}, we immediately obtain \begin{align*}
        \frac{\partial}{\partial t}\widetilde{\Gamma}(t,u)&=\frac{\d}{\d t}\widetilde{\Gamma}^{u_0(t,u)}(t)\\
        &=\frac{\partial}{\partial s}\widetilde{\Gamma}^{u_0(t,u)}(t)+\frac{\partial}{\partial u_0}\widetilde{\Gamma}^{u_0(t,u)}(t)\frac{\partial}{\partial t}u_0(t,u)\\
        &=\widetilde{\Gamma}_u^{u_0(t,u)}(t)\frac{\partial}{\partial s}u^{u_0(t,u)}(t)+\widetilde{\Gamma}_t^{u_0(t,u)}(t)+\widetilde{\Gamma}_u^{u_0(t,u)}(t)\frac{\partial}{\partial u_0}u^{u_0(t,u)}(t)\frac{\partial}{\partial t}u_0(t,u)\\
        &=\widetilde{\Gamma}_t^{u_0(t,u)}(t)+\widetilde{\Gamma}_u^{u_0(t,u)}(t)\frac{\d}{\d t}u^{u_0(t,u)}(t)\\
        &=\widetilde{\Gamma}_t^{u_0(t,u)}(t).
    \end{align*}
    In the same manner, $\frac{\partial}{\partial u}\widetilde{\Gamma}(t,u)=\widetilde{\Gamma}_u^{u_0(t,u)}(t)$ can also be verified. The uniqueness of $C^2$-classical solution follows from the procedure in deriving the characteristic system; see, e.g.,  \citet[Section 3.2.1]{Evans2010} for a rigorous argument.

    \section{Proof of Theorem \ref{verification_2}}\label{verif_proof}
    Clearly, $(k^*)'=(k')^{-1}\in C^1[0,+\infty)$ is locally Lipschitz continuous. Using the fact that $\Gamma_u\in C^1([0,T)\times(-\frac{1}{2\sigma_0^2},+\infty))\cap C([0,T]\times(-\frac{1}{2\sigma_0^2},+\infty))$, it is easy to verify the local Lipschitz continuity of the right-hand side of Eq.\eqref{41}. Thus the existence and uniqueness of local solution follows again from Picard-Lindel\"of theorem. Observe from the system \eqref{31} that for every $(t,u)\in [0,T]\times[0,+\infty)$: \begin{equation}\label{partial_derivative_estimation}
            0< \frac{\sigma_0^4}{2\sigma^2\gamma}\frac{T-t}{(\sigma_0^2u_0+1)^2}\le -\frac{1}{\gamma}\Gamma_u(t,u)\le \frac{\sigma_0^4}{2\sigma^2\gamma}\frac{T-t}{(\sigma_0^2u+1)^2}\le \frac{\sigma_0^4(T-t)}{2\sigma^2\gamma}.
        \end{equation}
        We conclude that $0\le Z_t\le (\frac{1}{\sigma^2}+(k')^{-1}(\frac{\sigma_0^4T}{2\sigma^2\gamma}))T$  for all $t\in [0,T]$. Therefore, the local solution is well defined on the whole time interval $[0,T]$.
        
         Recall that the optimizers to HJB equation \eqref{HJB} have the following feedback form: \begin{align*}{\pi}(t,x,y,z)&=-\frac{V_{xy}(t,x,y,z)}{\sigma^2V_{xx}(t,x,y,z)}=\frac{1}{\sigma^2\gamma}\frac{\mu_0+\sigma_0^2y}{\sigma_0^2z+\frac{\sigma_0^2}{\sigma^2}(T-t)+1},\\
        ({\theta}^2)(t,x,y,z)&=\underset{\theta>0}{\text{argmax}}\;\left\{\theta(V_z(t,x,y,z)+\frac{1}{2}V_{yy}(t,x,y,z))-k(\theta)V_x(t,x,y,z)\right\}\\
        &= (k')^{-1}\left(\frac{V_z(t,x,y,z)+\frac{1}{2}V_{yy}(t,x,y,z)}{V_x(t,x,y,z)}\right)\\
        &=(k')^{-1}\left(-\frac{1}{\gamma}\Gamma_u(t,z+\frac{1}{\sigma^2}(T-t))\right).
    \end{align*}
    From these expressions we can easily obtain Eqs \eqref{42}--\eqref{43}, and the verification result is proved by using the same way as in Theorem \ref{verification_theorem}.

    \section{Proof of Theorem \ref{limit}}\label{limit_proof}
     Let $(\vartheta^{(b)},\varpi^{(b)})$ be the optimal information-and-trading strategy for cost function $k^{(b)}$ and define $$u^{(b)}_t=\frac{T}{\sigma^2}+\int_0^t(\vartheta^{(b)}_s)^2 \d s.$$ 
     From Remark \ref{det_opti} and Section \ref{eco_int}, we see that $\vartheta^{(b)}$ solves the deterministic optimal control problem \begin{align}\label{inner_problem2}
        \nonumber \mini_{\vartheta\;\text{bounded}}&\;\left\{ \int_0^T\gamma k^{(b)}(\vartheta_t^2)\d t+\int_0^T\frac{\sigma_0^2}{2\sigma^2}\frac{1}{\sigma_0^2u_t+1}\d t-\frac{\mu_0^2}{2\sigma_0^2}\right\}\\
        &\text{subject to }
         \quad\d u_t=\vartheta_t^2\d t,\;u_0=\frac{T}{\sigma^2},
        \end{align}
    and $u^{(b)}$ is the corresponding state process. The above optimization problem is strictly convex and thus admits at most one optimal solution. Therefore, $\vartheta^{(b)}$ is actually the unique optimizer. Based on the classical variational theory (cf. \citet[Section 1.10]{benton1977}),  we know that $\vartheta^{(b)}$ satisfies \begin{equation}\label{Euler}
        -\gamma (k^{(b)})''((\vartheta^{(b)}_t)^2)\frac{\d (\vartheta^{(b)}_t)^2}{\d t}=\frac{\sigma_0^4}{2\sigma^2}\frac{1}{(\sigma_0^2 u^{(b)}_t+1)^2},
    \end{equation}
    which is the Euler-Lagrange equation related to the optimization problem \eqref{inner_problem2}.
    
    Consider $b\in (0,1]$ only. Let $\Gamma^{(b)}$ be the solution to Eq.\eqref{29} and $Z^{(b)}$ be the solution to Eq.\eqref{41} with $k$ replaced by $k^{(b)}$. Define $u_0^{(b)}(t,u)$ similarly ($u_0$ is introduced in the Appendix \ref{GC_proof}). From the Hamilton's ODE \eqref{31} we immediately see that for each fixed $u\in [0,+\infty)$, there exists $\overline{u}(u)>0$ such that $u_0^{(b)}(t,u)<\overline{u}(u)$ holds for every $(t,b)\in [0,T]\times(0,1]$.
    Therefore, using Eq.\eqref{42} and the derivative estimate \eqref{partial_derivative_estimation}, we have \begin{align*}
            (\vartheta_t^{(b)})^2&=\left((k^{(b)})'\right)^{-1}\left(-\frac{1}{\gamma}\Gamma_u^{(b)}\left(t,Z_t^{(b)}+\frac{1}{\sigma^2}(T-t)\right)\right)\\
            &\in \left[\left((k^{(1)})'\right)^{-1}\left(\frac{\sigma_0^4}{2\sigma^2\gamma}\frac{T-t}{\left(\sigma_0^2\overline{u}\left(\frac{2T}{\sigma^2}+T(k')^{-1}(\frac{\sigma_0^4T}{2\sigma^2\gamma})\right)+1\right)^2}\right),\left(k'\right)^{-1}\left(\frac{\sigma_0^4(T-t)}{2\sigma^2\gamma}\right)\right]
        \end{align*}
        for every $b\in (0,1]$. By  Eq.\eqref{Euler}, we see that, on each closed sub-interval $[0,\tau]\subset[0,T)$, both $ (\vartheta^{(b)}_t)^2$ and $\frac{\d (\vartheta^{(b)}_t)^2}{\d t}$ are bounded, uniformly in $b\in (0,1]$. Therefore,  taking $\tau=T-\frac{1}{n}$ and repeatedly applying the Ascoli-Arzela's lemma, we  obtain a decreasing sequence $\{b_n\}_{n\ge 1}$ with $b_n\rightarrow 0$ such that $\vartheta_t^{(b_n)}$ converge to some continuous function $\vartheta_t^{(0)}$, uniformly on compact subsets of $[0,T)$. 

        Define $\varpi^{(0)}$ as in the statement of the theorem. We show that $(\varpi^{(0)},\vartheta^{(0)})$ is actually an optimal realized strategy. Notice that the admissibility of information-and-trading strategy is independent of the cost function $k$. Consider any other realized strategy $(\vartheta,\varpi)$ and let $X,X^{(0)}$ be the corresponding wealth process generated by $(\vartheta,\varpi)$ and $(\varpi^{(0)},\vartheta^{(0)})$, respectively. If $E[U(X_T)]=-\infty$, the inequality $E[U(X_T)]\le E[U(X^{(0)}_T)]$ holds trivially. Otherwise, we have \begin{align*}
            \E&[U(X_T)]=\E\left[-\frac{1}{\gamma}\exp\left\{-\gamma\left(x+\int_0^T\varpi_t\mu\d t+\int_0^T\varpi_t\sigma\d W_t-\int_0^Tk(\vartheta_t^2)\d t\right)\right\}\right]\\
            &=\underset{n\rightarrow \infty}{\lim}\E\left[-\frac{1}{\gamma}\exp\left\{-\gamma\left(x+\int_0^T\varpi_t\mu\d t+\int_0^T\varpi_t\sigma\d W_t-\int_0^Tk^{(b_n)}(\vartheta_t^2)\d t\right)\right\}\right]\\
            &\le \underset{n\rightarrow \infty}{\limsup}\E\left[-\frac{1}{\gamma}\exp\left\{-\gamma\left(x+\int_0^T\varpi^{(b_n)}_t\mu\d t+\int_0^T\varpi^{(b_n)}_t\sigma\d W_t-\int_0^Tk^{(b_n)}((\vartheta^{(b_n)}_t)^2)\d t\right)\right\}\right]\\
            &= \underset{n\rightarrow \infty}{\limsup}\E\left[-\frac{1}{\gamma}\exp\left\{-\gamma\left(x+\int_0^T\varpi^{(b_n)}_t\mu\d t+\int_0^T\varpi^{(b_n)}_t\sigma\d W_t-\int_0^Tk((\vartheta^{(0)}_t)^2)\d t\right)\right\}\right]\\
            &\le \E\left[-\frac{1}{\gamma}\exp\left\{-\gamma\left(x+\int_0^T\varpi^{(0)}_t\mu\d t+\int_0^T\varpi^{(0)}_t\sigma\d W_t-\int_0^Tk((\vartheta^{(0)}_t)^2)\d t\right)\right\}\right]\\
            &=\E[U(X^{(0)}_T)],
            \end{align*}
        where in the second equality we used the dominated convergence theorem, in the first inequality we used the optimality of the realized strategy $(\vartheta^{(b_n)},\varpi^{(b_n)})$, in the third equality we used the fact that $\vartheta^{(b_n)}$ is deterministic, and in the second inequality we used the fact that $\varpi^{(0)}$ is the optimal response given information acquisition strategy $\vartheta^{(0)}$.

        Now we know that $(\varpi^{(0)},\vartheta^{(0)})$ is an optimal realized strategy. Therefore, according to our analysis in Section \ref{eco_int}, $\vartheta^{(0)}$ must be the optimal solution to Problem \eqref{inner_problem}. If the limiting function $\vartheta^{(0)}_t\triangleq \underset{b\rightarrow 0}{\lim}\vartheta^{(b)}_t$ is not well-defined, then there exists another sequence $\{\widetilde{b}_n\}_{n\ge 1}\rightarrow 0$ such that $\vartheta^{\widetilde{b}_n}$ converge to some continuous function $\widetilde{\vartheta}$ different from $\vartheta^{(0)}$. Repeat the above argument and we find that $\widetilde{\vartheta}$ also solves Problem  \eqref{inner_problem}, which contradicts the uniqueness of its optimal solution. 

    \section{Proof of Theorem \ref{property}}\label{property_proof}
    Property (a) follows directly from the derivative estimate \eqref{partial_derivative_estimation}. Property (b) is obtained by noticing that $\vartheta^*$ is actually the optimal solution to problem \eqref{inner_problem2}. If $\vartheta^*$ is not decreasing in $t$, then by its continuity we may find two closed intervals $A_1=[t_1,t_1+t_0]$ and $A_2=[t_2,t_2+t_0]$ such that $\inf\{\vartheta^*_t:t\in A_2\}> \sup\{\vartheta^*_t:t\in A_1\}$ and $t_1<t_2$. Switching the definition of $\vartheta^*$ on the two intervals then contradicts its optimality. To see (c), note that the Hamilton's ODE system \eqref{31} does not depend on the choice of $\mu_0$. This indicates that the characteristic curves, as well as the spatial derivative $\Gamma_u$, are unaffected by changes in $\mu_0$. The optimal information acquisition strategy, however, depends only on the spatial derivative $\Gamma_u$.

\end{appendices}

\bibliographystyle{abbrvnat}
\bibliography{0817}

\begin{thebibliography}{30}
\providecommand{\natexlab}[1]{#1}
\providecommand{\url}[1]{\texttt{#1}}
\expandafter\ifx\csname urlstyle\endcsname\relax
  \providecommand{\doi}[1]{doi: #1}\else
  \providecommand{\doi}{doi: \begingroup \urlstyle{rm}\Url}\fi

\bibitem[Andrei and Hasler(2020)]{Andrei2020}
D.~Andrei and M.~Hasler.
\newblock Dynamic attention behavior under return predictability.
\newblock \emph{Management Science}, 66\penalty0 (7):\penalty0 2906--2928,
  2020.

\bibitem[Bain and Crisan(2009)]{bain2009}
A.~Bain and D.~Crisan.
\newblock \emph{Fundamentals of Stochastic Filtering}.
\newblock Springer, New York, 2009.

\bibitem[Benton(1977)]{benton1977}
S.~H. Benton.
\newblock \emph{The Hamilton-Jacobi Equation: A Global Approach.}
\newblock Academic Press, New York, 1977.

\bibitem[Bismuth et~al.(2019)Bismuth, Gu{\'e}ant, and Pu]{Bis2019}
A.~Bismuth, O.~Gu{\'e}ant, and J.~Pu.
\newblock Portfolio choice, portfolio liquidation, and portfolio transition
  under drift uncertainty.
\newblock \emph{Mathematics and Financial Economics}, 13\penalty0 (4):\penalty0
  661--719, 2019.

\bibitem[Brendle(2006)]{brendle2006}
S.~Brendle.
\newblock Portfolio selection under incomplete information.
\newblock \emph{Stochastic processes and their Applications}, 116\penalty0
  (5):\penalty0 701--723, 2006.

\bibitem[Cohen et~al.(2025)Cohen, Knochenhauer, and Merkel]{cohen2025}
S.~N. Cohen, C.~Knochenhauer, and A.~Merkel.
\newblock Optimal adaptive control with separable drift uncertainty.
\newblock \emph{SIAM Journal on Control and Optimization}, 63\penalty0
  (2):\penalty0 1348--1373, 2025.

\bibitem[Danilova et~al.(2010)Danilova, Monoyios, and Ng]{danilova2010}
A.~Danilova, M.~Monoyios, and A.~Ng.
\newblock Optimal investment with inside information and parameter uncertainty.
\newblock \emph{Mathematics and Financial Economics}, 3:\penalty0 13--38, 2010.

\bibitem[Davis and Lleo(2013)]{davis2013}
M.~Davis and S.~Lleo.
\newblock Black--litterman in continuous time: the case for filtering.
\newblock \emph{Quantitative Finance Letters}, 1\penalty0 (1):\penalty0 30--35,
  2013.

\bibitem[Detemple(1986)]{detemple1986}
J.~B. Detemple.
\newblock Asset pricing in a production economy with incomplete information.
\newblock \emph{The Journal of Finance}, 41\penalty0 (2):\penalty0 383--391,
  1986.

\bibitem[Evans(2010)]{Evans2010}
L.~C. Evans.
\newblock \emph{Partial Differential Equations}.
\newblock American Mathematical Society, 2nd edition, 2010.

\bibitem[Fleming and Soner(2006)]{Fleming2006}
W.~H. Fleming and H.~M. Soner.
\newblock \emph{Controlled Markov Processes and Viscosity Solutions}.
\newblock Springer, New York, 2nd edition, 2006.

\bibitem[Frey et~al.(2012)Frey, Gabih, and Wunderlich]{frey2012}
R.~Frey, A.~Gabih, and R.~Wunderlich.
\newblock Portfolio optimization under partial information with expert
  opinions.
\newblock \emph{International Journal of Theoretical and Applied Finance},
  15\penalty0 (01):\penalty0 1250009, 2012.

\bibitem[Gargano and Rossi(2018)]{gargano2018}
A.~Gargano and A.~G. Rossi.
\newblock Does it pay to pay attention?
\newblock \emph{The Review of Financial Studies}, 31\penalty0 (12):\penalty0
  4595--4649, 2018.

\bibitem[Gennotte(1986)]{gennotte1986}
G.~Gennotte.
\newblock Optimal portfolio choice under incomplete information.
\newblock \emph{The Journal of Finance}, 41\penalty0 (3):\penalty0 733--746,
  1986.

\bibitem[Georgiou and Lindquist(2013)]{georgiou2013}
T.~T. Georgiou and A.~Lindquist.
\newblock The separation principle in stochastic control, redux.
\newblock \emph{IEEE Transactions on Automatic Control}, 58\penalty0
  (10):\penalty0 2481--2494, 2013.

\bibitem[Guan et~al.(2025)Guan, Liang, and Xia]{Guan2025}
G.~Guan, Z.~Liang, and J.~Xia.
\newblock Equilibrium portfolio selection for smooth ambiguity preferences.
\newblock \emph{Mathematics of Operations Research}, 50\penalty0 (2):\penalty0
  1042--1071, 2025.

\bibitem[Guiso and Jappelli(2020)]{guiso2020}
L.~Guiso and T.~Jappelli.
\newblock Investment in financial information and portfolio performance.
\newblock \emph{Economica}, 87\penalty0 (348):\penalty0 1133--1170, 2020.

\bibitem[Gupta-Mukherjee and Pareek(2020)]{gupta2020}
S.~Gupta-Mukherjee and A.~Pareek.
\newblock Limited attention and portfolio choice: The impact of attention
  allocation on mutual fund performance.
\newblock \emph{Financial Management}, 49\penalty0 (4):\penalty0 1083--1125,
  2020.

\bibitem[Hale(1980)]{Hale1980}
J.~K. Hale.
\newblock \emph{Ordinary Differential Equations}.
\newblock Krieger, Malabar, 2nd edition, 1980.

\bibitem[Karatzas and Shreve(1991)]{Karatzas1991}
I.~Karatzas and S.~Shreve.
\newblock \emph{Brownian Motion and Stochastic Calculus}.
\newblock Springer, New York, 2nd edition, 1991.

\bibitem[Karatzas and Zhao(2001)]{karatzas2001}
I.~Karatzas and X.~Zhao.
\newblock Bayesian adaptive portfolio optimization.
\newblock In E.~Jouini, J.~Cvitanic, and M.~Musiela, editors, \emph{Handbooks
  in Mathematical Finance: Option Pricing, Interest Rates and Risk Management},
  pages 632--669, Cambridge, 2001. Cambridge University Press.

\bibitem[Liptser and Shiryaev(2001)]{liptser1977statistics}
R.~S. Liptser and A.~N. Shiryaev.
\newblock \emph{Statistics of Random Processes: I. General Theory}.
\newblock Springer, Berlin, 2nd edition, 2001.

\bibitem[Ma{\'c}kowiak et~al.(2023)Ma{\'c}kowiak, Mat{\v{e}}jka, and
  Wiederholt]{mackowiak2023}
B.~Ma{\'c}kowiak, F.~Mat{\v{e}}jka, and M.~Wiederholt.
\newblock Rational inattention: A review.
\newblock \emph{Journal of Economic Literature}, 61\penalty0 (1):\penalty0
  226--273, 2023.

\bibitem[Monoyios(2009)]{monoyios2009}
M.~Monoyios.
\newblock Optimal investment and hedging under partial and inside information.
\newblock \emph{Advanced Financial Modelling, Radon Series on Computational and
  Applied Mathematics}, 8:\penalty0 371--410, 2009.

\bibitem[Rieder and B{\"a}uerle(2005)]{Rieder2005}
U.~Rieder and N.~B{\"a}uerle.
\newblock Portfolio optimization with unobservable {M}arkov-modulated drift
  process.
\newblock \emph{Journal of Applied Probability}, 42\penalty0 (2):\penalty0
  362--378, 2005.

\bibitem[Rishel(1999)]{rishel1999}
R.~Rishel.
\newblock Optimal portfolio management with partial observations and power
  utility function.
\newblock In \emph{Stochastic Analysis, Control, Optimization and Applications:
  A Volume in Honor of WH Fleming}, pages 605--619. Springer, 1999.

\bibitem[Sass and Haussmann(2004)]{sass2004}
J.~Sass and U.~G. Haussmann.
\newblock Optimizing the terminal wealth under partial information: The drift
  process as a continuous time markov chain.
\newblock \emph{Finance and Stochastics}, 8:\penalty0 553--577, 2004.

\bibitem[Sun and Guo(2015)]{Sun2015}
B.~Sun and B.~Z. Guo.
\newblock Convergence of an upwind finite-difference scheme for
  $\mathrm{H}$amilton-$\mathrm{J}$acobi-$\mathrm{B}$ellman equation in optimal
  control.
\newblock \emph{IEEE Transactions on Automatic Control}, 60\penalty0
  (11):\penalty0 3012--3017, 2015.

\bibitem[Wonham(1968)]{wonham1968}
W.~M. Wonham.
\newblock On the separation theorem of stochastic control.
\newblock \emph{SIAM Journal on Control}, 6\penalty0 (2):\penalty0 312--326,
  1968.

\bibitem[Yong and Zhou(1999)]{YongZhou1999}
J.~Yong and X.~Y. Zhou.
\newblock \emph{Stochastic Controls: Hamiltonian Systems and HJB Equations}.
\newblock Springer, New York, 1999.

\end{thebibliography}
\end{document}